\documentclass[11pt,reqno]{amsart}

\usepackage[text={150mm,232mm},centering]{geometry}   
\usepackage[mathscr]{eucal}
\usepackage[all]{xy}
\usepackage{mathrsfs,xypic}
\usepackage{amsfonts,amsmath,amsthm,amssymb,bm}
\usepackage{latexsym,tabularx,graphicx,pict2e}
\usepackage{pdfpages}
\usepackage{color}

\usepackage{hyperref}

\newtheorem{theorem}{Theorem}[section]
\newtheorem{lemma}[theorem]{Lemma}

\newtheorem{corollary}[theorem]{Corollary}
\newtheorem{proposition}[theorem]{Proposition}

\theoremstyle{definition}
\newtheorem{example}[theorem]{Example}
\newtheorem{definition}[theorem]{Definition}
\newtheorem{definition-lemma}[theorem]{Definition-Lemma}
\newtheorem{definition-theorem}[theorem]{Definition-Theorem}

\newtheorem{remark}[theorem]{Remark}

\newtheorem*{notation}{Notation}

\newcommand{\ac}{\textup{!`}}

\allowdisplaybreaks

\title[Batalin-Vilkovisky algebra structure]{Batalin-Vilkovisky algebra
structure on Poisson manifolds with diagonalizable modular symmetry}

\author{Xiaojun Chen}

\author{Leilei Liu}

\author{Sirui Yu}

\author{Jieheng Zeng}

\address{Chen, Yu: School of Mathematics, Sichuan University, Chengdu,
Sichuan Province, 610064 P.R. China}
\email{xjchen@scu.edu.cn, banaenoptera@163.com}

\address{Liu: School of Mathematics (Zhuhai), Sun Yat-sen University, Zhuhai,
Guangdong Province, 519082 P.R. China}
\email{liuleilei@mail.sysu.edu.cn}

\address{Zeng: School of Mathematics, Peking University, Beijing, 100871 P.R. China}
\email{zengjh662@163.com}

\date{}

\begin{document}

\begin{abstract}
We study the ``twisted" Poincar\'e duality
of smooth Poisson manifolds, and show that, if the modular vector field is diagonalizable, 
then there is a mixed complex
associated to the Poisson complex, which, combining with the twisted Poincar\'e duality,
gives a Batalin-Vilkovisky algebra structure on the Poisson
cohomology. This generalizes the previous results
obtained by Xu for unimodular Poisson manifolds.
We also show that the Batalin-Vilkovisky algebra structure is preserved under
Kontsevich's deformation quantization, and in the case of polynomial algebras
it is also preserved by Koszul duality.


\end{abstract}

\maketitle

\setcounter{tocdepth}{1}\tableofcontents

%=================================================

\section{Introduction}

Let $X$ be a smooth, oriented Poisson manifold.
Let $A$ be the algebra of smooth functions on $X$.
The Poisson cohomology and homology of $A$, and hence of $X$,
were introduced by Lichnerowicz \cite{Lichnerowicz}
and Koszul \cite{Koszul} respectively. They were further
studied by, for example, Brylinski \cite{Brylinski}, and Xu \cite{Xu}.
In particular, Xu found that there is an obstruction for the existence
of the Poincar\'e duality between the Poisson cohomology and homology
of $X$. Such an obstruction lies in the first Poisson
cohomology of $X$, called the {\it modular class}, and is represented by the
modular vector field
of the Poisson structure.
If the obstruction vanishes, in which case $X$ is called {\it unimodular},
then we have the Poincar\'e duality on $X$.
As a corollary, he showed that there exists a Batalin-Vilkovisky algebra
structure on the Poisson cohomology, which is nontrivial in general,
in the sense that the Batalin-Vilkovisky operator
generates the Schouten bracket.

The purpose of this paper is to generalize Xu's result to a class of
Poisson manifolds with non-trivial modular class, and then to study
some algebraic structures, such as the Batalin-Vilkovisky algebra
structure among others, associated to them.

\subsection{Poincar\'e duality for Poisson manifolds}
In 1998, Van den Bergh
studied in \cite{VdB97} the Poincar\'e duality
problem for associative algebras.
For an associative algebras, say $A$,
Van den Bergh showed that if $A$ is homologically smooth, then
there is an isomorphism between the Hochschild cohomology
of $A$ and the Hochschild homology of $A$ with values
in $A$ tensor with its inverse dualizing complex.
If the inverse dualizing complex is trivial, in which case
the algebra is called {\it Calabi-Yau}, then we have the Poincar\'e
duality between the Hochschild cohomology and homology of $A$.

In some cases that we are interested in, the associative algebras, such as
the Artin-Schelter
regular (AS-regular for short) algebras, are not Calabi-Yau,
but are very close to be so.
Inspired by noncommutative differential geometry
Brown and Zhang studied the ``twisted" Hochschild homology of
an AS-regular algebra, say $A$, and showed that Van den Bergh's Poincar\'e duality
has the form (see \cite{BZ} and also \cite{RRZ})
$$\mathrm{HH}^\bullet(A)\cong\mathrm{HH}_{n-\bullet}(A, A_\sigma),$$
where $\mathrm{HH}^\bullet(A)$ is the Hochschild cohomology of $A$
while $\mathrm{HH}_{\bullet}(A, A_\sigma)$
is the Hochschild homology of $A$
with coefficients in $A$ twisted with its {\it Nakayama automorphism},
and $n$ is the global dimension of $A$.
In this case we say $A$ admits the {\it twisted}
Poincar\'e duality.

Going back to the Poisson algebra
case, the twisted Poincar\'e duality was first studied
by Launois and Richard \cite{LR} for some quadratic Poisson algebras,
which was later generalized by Zhu in \cite{Zhu} and Luo, Wang and Wu in \cite{LWW}.
In 2017 L\"u, Wang and Zhuang obtained in \cite{LWZ} the twisted Poincar\'e duality theorem
for Poisson Calabi-Yau affine varieties, which covers all the above cases.

\begin{theorem}[Theorem \ref{thm:twistedPD1};
see also \cite{LWW}]\label{thm:main1}
Let $X$ be a smooth and oriented Poisson $n$-manifold with a fixed
volume form. Let $A$ be the
ring of smooth functions on $X$ and $\nu$ be the modular vector field.
Let $A_\nu$ be $A$ itself twisted with $\nu$ (see Example \ref{ex:Poissonmodules}(2)
below for the precise definition).
Then we have the twisted Poincar\'e duality
$$\mathrm{HP}^\bullet(A)\cong\mathrm{HP}_{n-\bullet}(A, A_\nu),$$
where $\mathrm{HP}^\bullet(-)$ and $\mathrm{HP}_\bullet(-)$
are the Poisson cohomology and homology functors.
\end{theorem}

\subsection{The Batalin-Vilkovisky algebra structure}
For an associative algebra $A$,
the Hochschild complex $(\mathrm{CH}_\bullet(A), b, B)$
is a mixed complex, where $b$ is the Hochschild differential
and $B$ is Connes' cyclic operator.
But for algebras
with an automorphism $\sigma$ such as
AS-regular algebras as above,
the Hochschild complex that we are interested in
is $\mathrm{CH}_\bullet(A, A_\sigma)$, which does not
admit a mixed complex structure since Connes' cyclic operator does
not commute with $b$ unless $\sigma=\mathrm{Id}$
(that is, $A$ is Calabi-Yau).
Nevertheless, there is a special class
of AS-regular algebras which do have a mixed complex structure on its
(sub but homotopy equivalent) twisted Hochschild complex.
These are the AS-regular algebras whose Nakayama automorphism
is {\it diagonalizable}. In this
case Kowalzig and Kr\"ahmer showed in \cite{KK} that
these algebras share even more features of Calabi-Yau algebras; for example,
their Hochschild cohomology has a nontrivial Batalin-Vilkovisky algebra structure
(in the Calabi-Yau case this
is proved by Ginzburg in \cite{Ginzburg}).
Batalin-Vilkovisky algebra arose from physics,
especially from string field theory and
topological conformal field theory (see, for example,
\cite{Getzler1}). They have been widely studied
in recent years, by both physicists and mathematicians.

Going back to the Poisson algebra case, the situation is similar.
Suppose $A$ is a Poisson algebra with
nontrivial modular vector field $\nu$, then
in general the twisted Poisson complex
$(\mathrm{CP}_{\bullet}(A, A_\nu), \partial, d)$
is not a mixed complex, where $\partial$ is the Poisson boundary
and $d$ is the de Rham differential.
If we view the modular vector field as the infinitesimal version of the Nakayama automorphism,
then in the diagonalizable case,
we again have a mixed complex structure on the (sub but homotopy equivalent)
twisted Poisson complex.
Together with the twisted Poincar\'e duality,
the pair
$(\mathrm{HP}^\bullet(A), \mathrm{HP}_{n-\bullet}(A, A_\nu))$
form the so-called {\it differential calculus with duality}
(a notion introduced by Lambre \cite{Lambre} based on Tamarkin and Tsygan \cite{TT}),
which leads to the following theorem and generalizes Xu's result in \cite{Xu}
where only unimodular Poisson manifolds are considered.

\begin{theorem}[Theorems \ref{thm:BVforPoisson}]
Let $X$ be a smooth and oriented Poisson $n$-manifold
with diagonalizable modular vector field.
Let $A$ the algebra of smooth functions on $M$.
Then $\mathrm{HP}^\bullet(A)$
has a Batalin-Vilkovisky algebra structure, whose Batalin-Vilkovisky
operator generates the Schouten bracket on
$\mathrm{HP}^\bullet(A)$.
\end{theorem}

\subsection{Koszul duality, and deformation quantization}

For a quadratic Poisson polynomial algebra,
Shoikhet \cite{Shoikhet} showed that its Koszul dual is graded Poisson
and Tamarkin's deformation quantizations of these two Poisson algebras, one
is AS-regular and the other is Frobenius,
are again Koszul dual to each other
as graded associative algebras.
Later this result is proved to be true for
Kontsevich's deformation quantization by Calaque et al. \cite{Calaque+}.

On the other hand, for an arbitrary
Poisson polynomial algebra, Dolgushev
\cite{Dolgushev}
proved that its deformation quantization is an
AS-regular
algebra; in particular,
if the Poisson algebra is unimodular, then its deformation quantization is
Calabi-Yau.

Based on these results among others, it is shown in \cite{CCEY,CEL} that for a
unimodular Poisson algebra
$A=\mathbb R[x_1,\cdots,x_n]$, if we denote
by $A^!$ the Koszul dual algebra of $A$,
and by $A_\hbar$ and $A^!_\hbar$ the
deformation quantizations of $A$ and $A^!$ respectively,
then the following diagram
$$
\xymatrixcolsep{3pc}
\xymatrix
{
\mathrm{HP}^\bullet(A[\![\hbar]\!])\ar[r]^\cong\ar[d]^\cong&\mathrm{HP}^\bullet(A^![\![\hbar]\!])\ar[d]^\cong\\
\mathrm{HH}^\bullet(A_\hbar)\ar[r]^\cong&\mathrm{HH}^\bullet(A^!_\hbar)
}$$
is commutative as isomorphisms
of Batalin-Vilkovisky algebras.

In this paper we show that the above result remains true
if the modular vector field of the Poisson algebra is
diagonalizable.

\begin{theorem}[Theorems \ref{thm:BViso}]\label{thm:main3}
Let $A=\mathbb R[x_1,\cdots, x_n]$ be a Poisson algebra
with diagonalizable modular vector field.
Let $A^!$ be the Koszul dual of $A$, and let $A_\hbar$ and $A^!_\hbar$
be the deformation quantization of $A$ and $A^!$ respectively. Then
the following
$$
\xymatrixcolsep{3pc}
\xymatrix
{
\mathrm{HP}^\bullet(A[\![\hbar]\!])\ar[r]^{\cong}\ar[d]^{\cong}&\mathrm{HP}^\bullet(A^![\![\hbar]\!])\ar[d]^{\cong}\\
\mathrm{HH}^\bullet(A_\hbar)\ar[r]^{\cong}&\mathrm{HH}^\bullet(A^!_\hbar)
}$$
is a commutative diagram of isomorphisms
of Batalin-Vilkovisky algebras.
\end{theorem}

The Batalin-Vilkovisky algebra
structures for AS-regular
algebras and for Frobenius algebras with diagonalizable Nakayama were
independently proved by Kowalzig and Kr\"ahmer \cite{KK}
and Lambre, Zhou and Zimmermann \cite{LZZ}
respectively; their isomorphism in the Koszul
case was proved by \cite{Liu}.
We recently learned that Wang, Wu, Zhou and Zhu
have obtained the same Batalin-Vilkovisky algebra structure
for Poisson algebras with a diagonalizable
 vector vector field
in an unpublished manuscript \cite{WWZZ}.
What is new in above theorem is
that we study these algebraic structures in the category of Poisson algebras,
and relate them via
deformation quantization; it also answers a question
raised in \cite[\S7.3]{CCEY} where the authors asked
whether these algebraic structures exist
for Poisson algebras admitting the twisted Poincar\'e duality.

The rest of the paper is devoted to the proof of the
above theorems. It is organized as follows.
In \S\ref{sect:modularsymmetry} we study with some details the modular vector
field of Poisson algebras and then study the twisted Poincar\'e duality for Poisson manifolds.
In \S\ref{sect:BVstructure} we study the Batalin-Vilkovisky algebra structure on the Poisson cohomology of Poisson
algebras with diagonalizable modular vector field.
In \S\ref{sect:deformationquantization}
we show that the Batalin-Vilkovisky
algebra structure is preserved under deformation quantization.
In \S\ref{sect:FrobPoisson} we study the Koszul duality of quadratic Poisson algebras,
which are Frobenius Poisson algebras; we then study their twisted Poincar\'e duality
as well as their deformation quantization.
In \S\ref{sect:PDKDDQ} we combine the above results and show
Theorem \ref{thm:main3}. In \S\ref{subsect:gravity} we also discuss an algebraic structure
(the gravity algebra) on the negative cyclic homology of Poisson algebras
with diagonalizable modular vector fields.

\begin{notation} Throughout this paper,
$k$ denotes a field of characteristic $0$. 
All tensors and Homs are over $k$ unless otherwise specified.
All algebras (resp. coalgebras) are unital and augmented
(resp. co-unital and co-augmented) over $k$.
A Poisson algebra $A$ with the Poisson structure $\pi$ is
denoted by $(A,\pi)$, or by $(A, \cdot, \{-,-\})$.
If $A$ is an associative algebra,
then $A^{op}$ is its opposite and
$A^e=A\otimes A^{op}$ is its envelope.
All complexes are graded such that the differential has degree $-1$;
for a cochain complex, it is viewed as a chain complex by negating the grading,
and it is cohomology $\mathrm{H}^\bullet(-)$ is given by $\mathrm{H}_{-\bullet}(-)$ of its negation.
\end{notation}

\newtheorem*{ack}{Acknowledgements}
\begin{ack}
 The authors are grateful
to Farkhod Eshmatov for many helpful conversations and to
IASM, Zhejiang University for offering an excellent working
condition during the preparation of the
paper. The first author also thanks G. Zhou for helpful communications
and for sharing their unpublished manuscript \cite{WWZZ}.
This paper is supported by NSFC (Nos. 11890660 and 11890663).
\end{ack}

\section{Modular class and the Poincar\'e duality}\label{sect:modularsymmetry}

In this section, we briefly go over the modular vector fields for Poisson
algebras, and discuss twisted Poincar\'e duality for Poisson manifolds.
The main result of this section is Theorem \ref{thm:twistedPD1}.

\begin{definition}\label{def:Poissonmod}
Suppose $A$ is a Poisson $k$-algebra.
A {\it left Poisson $A$-module}
is a $k$-vector space $M$ endowed with two bilinear maps
$\cdot$ and $\{-,-\}_M: A\otimes M\rightarrow M$ such that
\begin{enumerate}
\item[(1)] $(M, \cdot)$ is a left module over the commutative algebra $A$;
\item[(2)] $( M, \{-,-\}_M )$ is a left module over the Lie algebra $(A, \{-,-\})$;
\item[(3)] $\{a, bx \}_M=\{a,b\}\cdot x+b\cdot \{a,x\}_M$  for any $a, b \in A$ and $x \in M$;
\item[(4)] $\{ab,x\}_M=a\cdot \{b,x\}_M+ b\cdot\{a,x\}_M$  for any $a, b \in A$ and $x \in M$.
\end{enumerate}
\end{definition}

The notion of right Poisson $A$-module is defined similarly, and is left to the reader.
A left Poisson $A$-module is not necessarily a right Poisson $A$-module; however,
for a right Poisson $A$-module $M$, if we denote its Lie action by $\{-,-\}_M$, then
it may be equipped with a left Poisson $A$-module, whose Lie action
is given by $a\otimes m\mapsto -\{m, a\}$, for all $a\in A$ and $m\in M$, and vice versa.
A {\it Poisson $A$-bimodule} is both a left and a right Poisson $A$-module
such that $\{a, m\}_M=-\{m, a\}_M$ for all $a\in A$ and $m\in M$.
In particular, $A$ itself is automatically a Poisson $A$-bimodule.

\begin{example}\label{ex:Poissonmodules}
(1) Suppose $M$ is a right (and respectively left) Poisson module over $A$.
Then its linear dual space $M^*:=\mathrm{Hom}_k(M, k)$ has a left (and respectively right)
Poisson module structure over $A$, with the dot product and the bracket adjoint to the product
and the bracket on $M$.
In particular, $A^*:=\mathrm{Hom}_k(A, k)$ is
both a right and a left Poisson $A$-module (in fact, a Poisson $A$-bimodule).

(2) Suppose $(M, \cdot, \{-,-\}_M)$ is a right Poisson $A$-module.
Let $\nu\in\mathfrak{X}^1(A)$ be a Poisson derivation; that is, a derivation of $A$
which commutes with the Poisson structure.
Define a new bracket $\{-,-\}_{M_\nu}: M\otimes A\rightarrow M$ by
\begin{equation}\label{eq:twistedbracket}
\{m,a\}_{M_{\nu}}=\{m,a\}_M+m\cdot \nu(a),
\end{equation}
for all $a\in A, m\in M$.
Then $(M,\cdot, \{-,-\}_{M_\nu})$ is again a right Poisson $A$-module,
called the {\it twisted Poisson $A$-module} twisted by the Poisson derivation $\nu$;
in what follows, we denote it by $M_\nu$.
Similarly, for a left Poisson $A$-module, we denote
the corresponding twisted Poisson $A$-module by $_{\nu}M$.
\end{example}

\begin{definition}[Lichnerowicz \cite{Lichnerowicz}]
Suppose $(A,\pi)$ is a Poisson algebra and $M$ is a left Poisson $A$-module.
Let
$
\mathfrak X^{p}_A(M)
$
be the space of skew-symmetric multilinear maps
$A^{\otimes p}\to M$ that are derivations in each argument; that is, the space
of $p$-th polyvectors on $A$ with values in $M$.
The {\it Poisson cochain complex} of $A$ with values in $M$,
denoted by $\mathrm{CP}^\bullet(A, M)$,
is the cochain complex
$$
\xymatrix{
M=\mathfrak X^0_A(M)\ar[r]^-{\delta^\pi}&\cdots\ar[r]&\mathfrak X^{p-1}_A(M)\ar[r]^-{\delta^\pi}&
\mathfrak X^{p}_A(M)\ar[r]^-{\delta^\pi}&\mathfrak X^{p+1}_A(M)\ar[r]^-{\delta^\pi}&\cdots
}$$
where
$\delta^\pi$ is given by
\begin{eqnarray*}
\delta^\pi(P)(f_0, f_1,\cdots, f_p)&:=&\sum_{0\le i\le p}(-1)^i\{f_i, P(f_0,\cdots, \widehat{f_i},\cdots, f_p)\}\\
&+&\sum_{0\le i<j\le p}(-1)^{i+j}P(\{f_i, f_j\}, f_0, \cdots, \widehat{f_i},\cdots, \widehat{f_j},\cdots, f_p),
\end{eqnarray*}
and $\widehat{\;}$ means the corresponding item is omitted.
The associated cohomology is called the {\it Poisson cohomology} of $A$ with values in $M$,
and is denoted
by $\mathrm{HP}^\bullet(A; M)$.
In particular, if $M=A$, then the cohomology is just called the {\it Poisson cohomology} of $A$, and is simply
denoted by $\mathrm{HP}^\bullet(A)$.
\end{definition}

\begin{definition}[Koszul \cite{Koszul}]
Suppose $(A,\pi)$ is a Poisson algebra and $N$ is a right Poisson $A$-module.
Denote by $\Omega^p_A(N)$ the set of $p$-th K\"ahler differential forms of $A$ with coefficients in $N$.
Then the  Poisson chain complex of $A$ with coefficients in $N$, denoted by $\mathrm{CP}_\bullet(A, N)$,
is
\begin{equation}
\xymatrix{
\cdots\ar[r]&\Omega^{p+1}_A(N)\ar[r]^{\partial_\pi}&\Omega^{p}_A(N)\ar[r]^{\partial_\pi}
&\Omega^{p-1}_A(N)\ar[r]^{\partial_\pi}&\cdots\ar[r]&\Omega^0_A(N)=N,
}
\end{equation}
where $\partial_\pi$ is given by
\begin{eqnarray*}
\partial_\pi(n\otimes df_1\wedge \cdots\wedge df_p)&=&\sum_{i=1}^p(-1)^{i-1}\{n, f_i\}_N\otimes
df_1\wedge\cdots \widehat{df_i}\cdots \wedge df_p\\
&+&\sum_{1\le i<j\le p}(-1)^{j-i}n\otimes
d\{f_i,f_j\}\wedge df_1\wedge\cdots \widehat{df_i}\cdots\widehat{df_j}\cdots \wedge df_p,
\end{eqnarray*}
where $n\in N$ and $f_1,\cdots, f_n\in A$.
The associated homology is called the {\it Poisson homology} of $A$ with coefficients
in $N$, and is denoted by
$\mathrm{HP}_\bullet(A, N)$. In particular, if $N=A$,
then the homology is just called the {\it Poisson homology} of $A$, and is simply
denoted by $\mathrm{HP}_\bullet(A)$.
\end{definition}

In what follows, if $\pi$ is clear from the text, we simply write
$\delta^\pi$ and $\partial_\pi$ as $\delta$ and $\partial$ respectively.
It should be noted that in both definitions, $\delta^\pi$ and $\partial_\pi$
are in fact the Lie derivative $L_\pi$ of $\pi$.
Suppose $A=\mathscr C^{\infty}(X)$ is the Poisson algebra
of the smooth functions on a smooth Poisson manifold,
or of the algebraic functions on a Poisson affine variety,
then $\mathrm{HP}^\bullet(A)$ and $\mathrm{HP}_\bullet(A)$
are Poisson invariants of $X$.

Suppose $\nu\in\mathfrak X^1(A)$ is a Poisson derivation,
then the chain complex $\mathrm{CP}_\bullet(A, A_\nu)$
has the same underlying vector space as $\mathrm{CP}_\bullet(A, A)$
but with the boundary,
which we now denote by $\partial_\nu$ in order to distinguish,
now becomes
\begin{eqnarray}
\partial_\nu(f_0\otimes df_1\wedge\cdots \wedge df_n)
&=&\partial(f_0\otimes df_1\wedge\cdots\wedge df_n)\nonumber\\
&+&\sum_{i=1}^n(-1)^{i-1}f_0\nu(f_i)\otimes df_1\wedge\cdots
\widehat{df_i}\cdots\wedge df_n,\label{formula:twistedboundary}
\end{eqnarray}
where $\partial$ is the boundary on $\mathrm{CP}_\bullet(A,A)$.

Now suppose we have an $n$-form
$\eta\in\Omega^n(A)$ such that the contraction
$$
\iota_{(-)}\eta:\mathfrak X_A^\bullet(A)\to\Omega_A^{n-\bullet}(A),\quad
X\mapsto\iota_X\eta
$$
is an isomorphism, then we say $\eta$ is a {\it volume form} of degree $n$.
If such a form $\eta$ exists, then we have the following diagram
\begin{equation}\label{diag:unimodularPoisson}
\begin{split}
\xymatrixcolsep{3pc}
\xymatrix{
\mathfrak{X}^\bullet_A(A)\ar[r]^-{\iota_{(-)}\eta}_{\cong}\ar[d]_{\delta}&\Omega_A^{n-\bullet}(A)\ar[d]^{\partial}\\
\mathfrak{X}^{\bullet+1}_A(A)\ar[r]^-{\iota_{(-)}\eta}_{\cong}&\Omega_A^{n-\bullet-1}(A),
}
\end{split}
\end{equation}
which is not necessarily commutative, since $\eta$ may not be a Poisson cycle.
To adjust this discrepancy, let us consider the following
commutative diagram
\begin{equation}\label{diag:unimodularPoisson111}
\begin{split}
\xymatrixcolsep{3pc}
\xymatrix{
\mathfrak{X}^\bullet_A(A)\ar[r]^-{\iota_{(-)}\eta}_{\cong}\ar[d]_{\mathrm{Div}}
&\Omega_A^{n-\bullet}(A)\ar[d]^{d}\\
\mathfrak{X}^{\bullet-1}_A(A)\ar[r]^-{\iota_{(-)}\eta}_{\cong}&\Omega_A^{n-\bullet+1}(A),
}
\end{split}
\end{equation}
where $\mathrm{Div}$ is the divergence operator.
Then
\begin{equation}\label{def:nu}
\nu:=-\mathrm{Div}(\pi)
\end{equation}
is a vector field,
and is called the {\it modular vector field} for $A$.
With these notations, we have the following proposition, which is due to Xu
(see \cite[Proposition 4.7]{Xu}):

\begin{proposition}[Xu]\label{TwistedPoincare}
Suppose $(A, \pi)$ is a Poisson algebra and $\eta$ is a volume form.
Then for any $\varphi\in \mathfrak{X}_A^p(A)$, we have
\begin{equation}\label{eq:Xusidentity}
(-1)^{\vert\varphi\vert-1}\partial (\iota_{\varphi}\eta)
-\iota_{\delta (\varphi)}\eta
=
\iota_\nu(\iota_\varphi\eta).
\end{equation}
\end{proposition}

\begin{proof}
On one hand, if we denote $\dagger:=\iota_{(-)}\eta$, then
\begin{equation}\label{eq:nu}
\iota_{\nu}(\iota_{\varphi}\eta)
=\iota_{\varphi}(\iota_\nu\eta)
\stackrel{\eqref{def:nu}}=
\iota_{\varphi}(\dagger(-\dagger^{-1}\circ d\circ\dagger(\pi)))
=-\iota_{\varphi}(d\circ\dagger(\pi))=-\iota_{\varphi}d\circ \iota_\pi\eta
=-\iota_\varphi\partial \eta,
\end{equation}
where the last equality holds
due to the Cartan formula $L_\pi=[\iota_\pi, d]$
and $\eta$ being $d$-closed.

On the other hand, we always have the equality
\[(-1)^{\vert\varphi\vert-1}\partial(\iota_{\varphi} \eta)
-\iota_{\delta(\varphi)}\eta=
-\iota_\varphi\partial\eta\]
Plugging \eqref{eq:nu} into the above identity, we get the desired equality.
\end{proof}

As an immediate corollary, we have the following ``twisted Poincar\'e duality":

\begin{theorem}[see also\cite{LWZ}]\label{thm:twistedPD1}
Let $A$ be a Poisson algebra with a volume form of degree $n$. Then
\begin{eqnarray}\label{eq:twistedPD1}
\mathrm{HP}^\bullet(A)\cong\mathrm{HP}_{n-\bullet}(A,A_\nu),
\end{eqnarray}
which is called the $\bm{twisted}$ Poincar\'e duality of $A$.
In particular, if $A$ is the set of smooth
functions on a smooth and oriented Poisson manifold,
or the set of algebraic functions of a Poisson Calabi-Yau affine
variety, then \eqref{eq:twistedPD1} holds.
\end{theorem}

\begin{proof}
In the light of \eqref{diag:unimodularPoisson}
and \eqref{eq:Xusidentity}, we only need to show
$$
(\partial+\iota_\nu)(\omega)=\partial_\nu(\omega),
$$
for any $\omega\in\Omega^\bullet(A)$.
This is a tautology by \eqref{formula:twistedboundary}.

Now, since for $A=\mathscr C^\infty(X)$
of a smooth and oriented Poisson manifold, or $A=\mathscr O(X)$
of a Poisson Calabi-Yau affine variety, then the volume form of $X$ (or say on $A$)
always exists, by the above argument, the theorem now follows.
\end{proof}

\begin{remark}
For a smooth and oriented Poisson manifold $X$,
the modular vector field $\nu$ for $A=\mathscr C^\infty(X)$
is a Poisson 1-cocycle, and the cohomology class
it represents does not depend on the choice of the volume form,
and hence is a topological invariant of the Poisson manifold,
which is usually called the {\it modular class} of $X$ (see \cite{Xu}
for more details).
For Poisson Calabi-Yau affine varieties, if we change the volume
form up to a unit, then the modular vector fields differ by
a {\it log-Hamiltonian derivation} (see \cite{Dolgushev} for more details).
\end{remark}

\begin{remark}\label{rmk:history}
Xu first studied the Poincar\'e duality
for Poisson manifolds (see \cite{Xu}),
where he showed that for a unimodular Poisson
manifold $X$,
$\mathrm{HP}^\bullet(X)\cong\mathrm{HP}_{n-\bullet}(X)$.
Later Launois and Richard in \cite{LR},
Zhu in \cite{Zhu}, and Luo, Wang and Wu in \cite{LWW}
studied the twisted Poincar\'e duality for some special types of Poisson
algebras. All these results are covered by the result
of L\"u, Wang and Zhuang  \cite[Corollary 4.4]{LWZ},
which deals with arbitrary Poisson Calabi-Yau affine varieties, with
a slightly different proof.
\end{remark}

\section{Modular vector fields and the Batalin-Vilkovisky structure}\label{sect:BVstructure}

Xu proved in \cite{Xu} that for unimodular
Poisson algebras, there exists
a Batalin-Vilkovisky algebra structure on its cohomology.
In this section we generalize this result to
Poisson algebras with {\it diagonalizable} modular vector fields.

\begin{definition}[Batalin-Vilkovisky algebra]
Suppose $(V,\bullet )$
is a graded commutative algebra. A {\it Batalin-Vilkovisky
algebra} structure on $V$ is the triple $(V,\bullet , \Delta)$
such that
\begin{enumerate}
\item[$(1)$] $\Delta: V^{i}\to V^{i-1}$ is a differential, that is, $\Delta^2=0$; and
\item[$(2)$] $\Delta$ is second order operator,
  that is,
 \begin{eqnarray*}
 \Delta(a\bullet  b\bullet  c)
 &=& \Delta(a\bullet  b)\bullet  c+(-1)^{|a|}a\bullet \Delta(b\bullet  c)+(-1)^{(|a|-1)|b|}b\bullet \Delta(a\bullet  c)\nonumber \\
 & & - (\Delta a)\bullet  b\bullet  c-(-1)^{|a|}a\bullet (\Delta b)\bullet  c-(-1)^{|a|+|b|}a\bullet  b\bullet  (\Delta c).
\end{eqnarray*}
\end{enumerate}
\end{definition}

In the above definition,
if we set
$$\{-,-\}: V\otimes V\to V,\; (a,b)\mapsto (-1)^{|a|}(\Delta(a\bullet b)-\Delta(a)\bullet
b-(-1)^{|a|}a\bullet \Delta(b)),\quad a, b\in V$$
then it is direct to check that
$(V,\cup,\{-,-\})$ is a Gerstenhaber algebra (see Definition \ref{def:Gerstenhaber}
below),
and we say the Gerstenhaber bracket $\{-,-\}$ is generated by
the Batalin-Vilkovisky operator $\Delta$ (see Getzler \cite{Getzler1} for more details).

Lambre observed in \cite{Lambre} that a lot of examples of Batalin-Vilkovisky
algebras come from the structure of differential calculus, in the sense of Tamarkin and Tsygan
\cite{TT}, with some additional conditions. Let us recall his result first.

\subsection{Differential calculus and the Batalin-Vilkovisky algebra}

We start with the notion of Gerstenhaber algebras:

\begin{definition}[Gerstenhaber]\label{def:Gerstenhaber}
A {\it Gerstenhaber algebra} is a graded $k$-vector space $A^{\bullet}$
endowed with two bilinear operators
$\cup: A^m \otimes A^n\rightarrow A^{m+n}$ and $\{-,-\}:A^n\otimes A^m\rightarrow A^{n+m-1}$
such that: for any homogeneous elements $a,b,c\in A^\bullet$,
\begin{enumerate}
\item[(1)] ($A^{\bullet}, \cup$) is a graded commutative associative algebra, i.e.,
$$a\cup b=(-1)^{|a||b|}b\cup a,$$ satisfying associativity;

\item[(2)] ($A^{\bullet},\{-,-\}$) is a graded Lie algebra with the bracket $\{-,-\}$ of degree $-1$, i.e.,
$$\{a,b\}=(-1)^{(|a|-1)(|b|-1)}\{b,a\}$$ and
$$\big\{a,\{b,c\}\big\}=\big\{\{a,b\},c\big\}+(-1)^{(|a|-1)(|b|-1)}\big\{b,\{a,c\}\big\};$$

\item[(3)] the cup product $\cup$ and the Lie bracket $\{-,-\}$ are compatible in the sense that
$$\{a,b\cup c\}=\{a,b\}\cup c+(-1)^{(|a|-1)|b|}b\cup\{a,c\}.$$
\end{enumerate}
\end{definition}

\begin{definition}[Tamarkin-Tsygan \cite{TT}, Definition 3.2.1]
Let $\mathrm{H}^\bullet$ and $\mathrm{H}_\bullet$ be two graded vector spaces.
A {\it differential calculus} is a sextuple
$$(\mathrm{H}^\bullet, \cup, \{-,-\}, \mathrm{H}_\bullet, \mathrm{B}, \cap),$$
satisfying the following conditions:
\begin{enumerate}
\item[(1)] $(\mathrm{H}^\bullet, \cup, \{-,-\})$ is a Gerstenhaber algebra;

\item[(2)] $\mathrm{H}_\bullet$ is a graded module over $(\mathrm{H}^\bullet,\cup)$ by the ``cap action"
$$\cap: \mathrm{H}^n\otimes \mathrm{H}_m \rightarrow \mathrm{H}_{m-n}, f\otimes \alpha \mapsto f\cap \alpha,$$
i.e.,
$(f\cup g)\cap \alpha=f\cap(g\cap \alpha)$ for any homogeneous $f, g\in \mathrm{H}^\bullet$, $\alpha \in \mathrm{H}_\bullet$;

\item[(3)] there exists a linear operator $\mathrm B: \mathrm{H}_\bullet\rightarrow \mathrm{H}_{\bullet+1}$
such that $\mathrm B^2=0$ and moreover, if we set
$L_f(\alpha):=\mathrm B(f\cap \alpha)-(-1)^{|f|}f\cap \mathrm B(\alpha)$,
then $L$ is a Lie algebra action of $\mathrm{H}^\bullet$ on $\mathrm H_\bullet$, that is,
$$L_{\{f,g\}}(\alpha)=[L_f,L_g](\alpha),$$
for any $f, g\in \mathrm H^\bullet$ and $\alpha \in\mathrm H_\bullet$.
\end{enumerate}
\end{definition}

\begin{example}\label{ex:differentialcalculusonPoisson}
Let $A$ be a Poisson algebra. Then
$$
(\mathrm{HP}^\bullet (A), \mathrm{HP}_\bullet(A, A),\cup,\{-,-\}, \cap, d)
$$
form a differential calculus,
where $\cup$ are $\{-,-\}$ are the wedge product
and the Schouten bracket induced
on the polyvectors, and $\cap$ is the contraction (also
denoted by $\iota$ before), and $d$ is the de Rham differential.
The key point here is to check that these operators are compatible with
the Poisson boundary and coboundary maps, which, however,
is a direct check; see \cite[Chapter 3]{LGPV} for more details.
\end{example}

\begin{definition}[Lambre \cite{Lambre}]
A differential calculus $(\mathrm{H}^\bullet,\mathrm{H}_\bullet,\cup, \{-,-\}, \cap,\mathrm B)$
is called a {\it differential calculus with duality} if there exists an element
$\eta\in\mathrm H_n$ for some $n$
such that
$$\phi:\mathrm{H}^{\bullet}\rightarrow \mathrm{H}_{n-\bullet},\;
\varphi\mapsto\varphi\cap\eta$$
is an isomorphism of $\mathrm H^\bullet$-modules.
\end{definition}

\begin{theorem}[Lambre \cite{Lambre}  Lemma 1.5 and Theorem 1.6]\label{calwithdualityinducesBV}
Assume $(\mathrm{H}^\bullet,\mathrm{H}_\bullet,\cup, \{-,-\}, \cap,\mathrm B)$
is a differential calculus with duality.
Let $\Delta:=\phi^{-1}\circ \mathrm{B}\circ \phi$. Then
$(\mathrm{H}^\bullet, \cup, \Delta)$ is a Batalin-Vilkovisky algebra where
$\Delta$ generates the Gerstenhaber bracket.
\end{theorem}

We next apply this theorem to the case of Poisson algebras.

\subsection{Poisson algebras with diagonalizable modular vector field}\label{subsect:diagonalizble}

Poisson structures with diagonalizable modular vector fields
are an important concept in Poisson geometry;
see, for example, \cite[\S8.2]{LGPV}, for more discussions.
In this subsection we show the existence of
a Batalin-Vilkovisky structure on the Poisson cohomology of a
Poisson algebra or a Frobenius Poisson algebra with
diagonalizable modular vector field.

For a Poisson algebra $A$ over $k$, 
a derivation $\nu\in\mathrm{Der}_k(A)$ is called {\it diagonalizable}
(or sometimes {\it semi-simple})
if there is a subset $\Sigma\subseteq k$
and a decomposition of $k$-vector spaces
$$A=\bigoplus_{\lambda\in\Sigma}A_\lambda,\quad
A_\lambda=\{a\in A\,|\,\nu(a)=\lambda a\}.$$
Now for a Poisson algebra with a diagonalizable modular vector field, we may decompose
its Poisson chain and cochain complexes into the direct sum of eigenspaces, which
leads to interesting results as we shall show below. We learned this idea from \cite{KK}
(see also \cite{LZZ,Liu} for some further applications).

Suppose $A$ has a diagonalizable modular vector field;
then we can decompose $A$ into the direct sum of eigenspaces of $\nu$, namely,
$A=\oplus_{\lambda_i}A_{\lambda_i}$,
where $A_{\lambda_i}:=\{a\in A| \nu(a)=\lambda_i a\}$. Let
$$
\mathrm{CP}_n^\lambda(A,A_\nu)
:=\left\{\sum f_0df_1\wedge\cdots \wedge df_n\in \mathrm{CP}_n(A,A_\nu)\Big|
\begin{array}{l}f_i\in A_{\lambda_i}\,\mbox{for some $\lambda_i$},\\
 i=0,1,\cdots, n,
\sum_{i=0}^n \lambda_i=\lambda
\end{array}
\right\}.
$$
Since $\nu(\{f,g\})=\{\nu(f),g\}+\{f,\nu(g)\}$, $\partial_\nu$ is closed on these spaces,
and hence $(\mathrm{CP}_\bullet^\lambda(A, A_\nu)
:=\bigoplus_n \mathrm{CP}_n^\lambda(A,A_\nu),\partial_\nu)$ is a subcomplex.
We have
\begin{equation}\label{eq:decompforchain}
\mathrm{CP}_\bullet(A,A_\nu)=\bigoplus_{\lambda}\mathrm{CP}_\bullet^\lambda(A, A_\nu).
\end{equation}
For the Poisson cochain complex $\mathrm{CP}^\bullet(A, A)$, we analogously have
a decomposition into the direct sum of the eigenspaces
\begin{equation}\label{eq:decompforcochain}
\mathrm{CP}^\bullet(A, A)
=\bigoplus_\lambda\mathrm{CP}^\bullet_\lambda(A, A))
=\bigoplus_\lambda\bigoplus_n\mathrm{CP}^n_\lambda(A, A)),
\end{equation}
where $\mathrm{CP}^n_\lambda(A, A_\nu)
:=\{\phi\in\mathrm{CP}^n_\lambda(A,A_\nu)|
\phi(A_{\lambda_1}\otimes\cdots\otimes A_{\lambda_n})\subset A_{\lambda_1+\cdots + \lambda_n+\lambda} \}.$

\begin{lemma}\label{lemma:homotopy}
Suppose $A$ is a Poisson algebra with diagonalizable modular vector field, then
$$
\partial_\nu\circ d+d\circ\partial_\nu=\tilde \nu,
$$
where $\tilde \nu (f_0df_1\wedge\cdots
\wedge df_n)=\nu(f_0)df_1\wedge\cdots \wedge df_n
+\sum_{i=1}^n f_0df_1\wedge\cdots\wedge d(\nu(f_i))\wedge\cdots \wedge df_n$.
\end{lemma}

\begin{proof}
For any element $f_0df_1\wedge\cdots\wedge df_n\in \mathrm{CP}_n(A,A_\nu)$, we have
\begin{eqnarray*}
&&\partial_\nu\circ d(f_0df_1\wedge\cdots\wedge df_n)
=\partial_\nu( df_0\wedge df_1\wedge\cdots\wedge df_n)\\
&=& \sum_{i=0}^n(-1)^i \nu(f_i)df_0\wedge\cdots\wedge \widehat{df_i} \wedge\cdots\wedge df_n\\
&+&\sum_{i=1}^n(-1)^{i} d\{f_0,f_i\}\wedge df_1\wedge
\cdots \wedge\widehat{d f_i}\wedge\cdots \wedge df_n\\
&+& \sum_{0<i< j}(-1)^{i+j} d\{f_i,f_j\}\wedge df_0\wedge \cdots \wedge
\widehat {df_i}\wedge\cdots\wedge \widehat{df_j}\wedge\cdots \wedge df_n.
\end{eqnarray*}
and
\begin{eqnarray*}
&&d\circ\partial_\nu(f_0df_1\wedge\cdots\wedge df_n)\\
&=&d\Big(\sum_{i=1}^n(-1)^{i-1} (\{f_0,f_i\}+f_0\nu(f_i))df_1
\wedge\cdots\wedge \widehat{df_i}\wedge\cdots\wedge df_n  \\
&+& \sum_{0<i<j}(-1)^{i+j}f_0 d\{f_i,f_j\}\wedge df_1\wedge
\cdots \wedge \widehat{df_i}\wedge\cdots\wedge \widehat{df_j}\wedge\cdots\wedge df_n \Big)\\
&=& \sum_{i=1}^n (-1)^{i-1} d\{f_0,f_i\}\wedge df_1\wedge
\cdots \wedge \widehat{df_i}\wedge\cdots \wedge df_n\\
&+&\sum_{i=1}^n(-1)^{i-1}d(f_0\nu(f_i))\wedge df_1\wedge
\cdots\wedge \widehat{df_i}\wedge\cdots\wedge df_n\\
&+& \sum_{0<i<j}(-1)^{i+j}df_0\wedge d\{f_i,f_j\}\wedge
df_1\wedge \cdots \wedge \widehat{df_i}\wedge\cdots\wedge \widehat{df_j}\wedge\cdots \wedge df_n.
\end{eqnarray*}
Hence we have
\begin{eqnarray*}&&(\partial_\nu\circ d+d\circ \partial_\nu )(f_0df_1\wedge\cdots \wedge df_n)\\
&=&\nu(f_0)df_1\wedge\cdots \wedge df_n
+\sum_{i=1}^n f_0df_1\wedge\cdots\wedge d(\nu(f_i))\wedge\cdots \wedge df_n\\
&=&\tilde\nu(f_0df_1\wedge\cdots \wedge df_n).
\end{eqnarray*}
This completes the proof.
\end{proof}

\begin{theorem}
Suppose $A$ is a Poisson algebra with a diagonalizable modular vector field,
then
\begin{equation}\label{eq:homotopyequiv}
\mathrm{HP}_\bullet(A, A_\nu)=\mathrm{H}_\bullet(\mathrm{CP}_\bullet^0(A, A_\nu))
\quad
\mbox{and}\quad
\mathrm{HP}^\bullet(A, A)=\mathrm{H}^\bullet(\mathrm{CP}^\bullet_0(A, A)).
\end{equation}
In particular,
$$
(\mathrm{HP}^\bullet (A), \mathrm{HP}_\bullet(A,A_\nu),\cup,\{-,-\}, \iota, d)
$$
forms a differential calculus with duality.
\end{theorem}

\begin{proof}
First, we have
inclusions
$$
i:\mathrm{CP}^0_\bullet(A, A_\nu)\hookrightarrow\mathrm{CP}_\bullet(A, A_\nu)\quad
\mbox{and}\quad
i:\mathrm{CP}_0^\bullet(A, A)\hookrightarrow\mathrm{CP}^\bullet(A, A).$$
We claim that these are homotopy equivalences of chain complexes.
In fact, by \eqref{eq:decompforchain}
and \eqref{eq:decompforcochain} the homotopy inverses are given by the projections.
If we denote the projections by $p$, then
$p\circ i=\mathrm{Id}$.
Now by Lemma \ref{lemma:homotopy} we have
\begin{equation}\label{eq:mixedcomplex}
(\partial_\nu\circ d+d\circ\partial_\nu)|_{\mathrm{CP}_\bullet^\lambda(A, A_\nu)}
=\tilde\nu|_{\mathrm{CP}_\bullet^\lambda(A, A_\nu)}
=\lambda\cdot\mathrm{Id}|_{\mathrm{CP}_\bullet^\lambda(A, A_\nu)},
\end{equation}
which means for $\lambda\ne 0$, the de Rham differential $d$,
up to a scalar, gives a homotopy retracting between $\mathrm{Id}$ and $i\circ p$.
This means,
$
i:\mathrm{CP}^0_\bullet(A, A_\nu)\hookrightarrow\mathrm{CP}_\bullet(A, A_\nu)
$ and similarly,
$i:\mathrm{CP}_0^\bullet(A, A)\hookrightarrow\mathrm{CP}^\bullet(A, A)$,
are equivalences of chain complexes, and \eqref{eq:homotopyequiv} follows.

Observe that from \eqref{eq:mixedcomplex} we also get that
$$(\mathrm{CP}_\bullet^0(A, A_\nu), \partial_\nu, d)$$
forms a mixed chain complex.
Now denote by $\eta$ the volume form of $A$, which represents
an $n$-class in
$\mathrm{HP}_n^0(A,A_\nu)$
corresponding to eigenvalue $0$.
We have that the cap action $\iota_{(-)}\eta$ preserves the eigenvalue
$$
\iota_{(-)}\eta: \mathrm{HP}^\bullet_\lambda(A)
\rightarrow \mathrm{HP}_{n-\bullet}^\lambda(A,A_{\nu}).
$$
Combining it with the twisted Poincare duality
$\mathrm{HP}^\bullet(A)\cong\mathrm{HP}_{n-\bullet}(A,A_\nu)$,
we get that
$$
(\mathrm{HP}^\bullet (A), \mathrm{HP}_\bullet(A,A_\nu),\cup,\{-,-\}, \iota, d)
$$
is a differential calculus with duality.
\end{proof}

Combining the above theorem with Lambre's Theorem \ref{calwithdualityinducesBV},
we get the following:

\begin{theorem}\label{thm:BVforPoisson}
Suppose $A$ is a Poisson algebra with diagonalizable
modular vector field, then
$\mathrm{HP}^\bullet(A)$ has a Batalin-Vilkovisky
algebra structure where the Batalin-Vilkovisky operator generates
the Gerstenhaber bracket.
\end{theorem}

\section{Deformation quantization}\label{sect:deformationquantization}

In this section we study the deformation quantization of Poisson algebras with
nontrivial modular vector field.
The ground field $k$ in this section is taken to be $\mathbb R$.

Suppose $A$ is a Poisson algebra;
its {\it (formal) deformation quantization},
denoted by $A_{\hbar}$, is a $k[\![\hbar]\!]-$linear associative product (called the star-product)
on $A[\![\hbar]\!]$
\[ a\star b=a\cdot b+B_1(a,b)\hbar+B_2(a,b)\hbar^2+\ldots,\quad \mbox{for}\; a, b\in A\]
such that $B_i:A\otimes A\to A$ are bidifferential operators, satisfying
$$B_1(a, b)-B_1(b, a)=\{a,b\}.$$
In what follows, we also write $B_i(a, b)$ as $\star_i(a, b)$.

In \cite{Kontsevich}
Kontsevich showed that there is a one-to-one correspondence between
the equivalence classes of the star-products and
the equivalence classes of Poisson structures
$\pi_\hbar=\pi+\pi_1\hbar+\cdots$
on
$A[\![\hbar]\!]$.
He also constructed an explicit $L_\infty$-quasi-isomorphism
$$\mathcal U: T_{\mathrm{poly}}(A)\to D_{\mathrm{poly}}(A)$$
from the space of polyvector fields to the Hochschild cochain complex which
acts on each component in $A$ as multi-derivations,
where the first term of $\mathcal U$ is the classical Hochschild-Kostant-Rosenberg
quasi-isomorphism.
Via this map, the Poisson bivector $\pi_\hbar$ on $A[\![\hbar]\!]$
corresponds to a star-product $\star$ on $A_{\hbar}$.
By considering the {\it tangent map} of $\mathcal U$, one then gets
a quasi-isomorphism
\begin{equation}\label{eq:isoofPoissonandHochschild}
\mathrm{CP}^\bullet(A[\![\hbar]\!],\pi_\hbar)\cong \mathrm{CH}^{\bullet}(A_{\hbar},\star).
\end{equation}
The reader may refer to Kontsevich's paper \cite{Kontsevich} for a proof
(see also Manchon-Torossian \cite{MT} for more details).

Later Dolgushev showed in \cite{Dolgushev}
that the deformation quantization of a Poisson polynomial
algebra is an AS-regular algebra;
similarly, the deformation quantization of a Poisson exterior
algebra is a graded Frobenius algebra.

What we are interested in now is to study the behavior
of the twisted Poisson homology $\mathrm{HP}_\bullet(A, A_\nu)$
under deformation quantization.

\subsection{Deformation quantization of Poisson bimodules}
We now
briefly go over the deformation quantization of Poisson bimodules.

\begin{definition}[Bursztyn-Waldmann \cite{BW}]
Suppose $M$ is a Poisson $A$-bimodule.
Suppose $A$ has a deformation quantization $A_\hbar$.
A {\it deformation quantization} of $M$,
denoted by $M_\hbar$, is $M[\![\hbar]\!]$ equipped with
an $A_\hbar$-bimodule structure such that
\begin{equation}\label{eq:infdeformationofmodules}
a\star_1m-m\star_1 a=\hbar\{a, m\},\quad \mbox{for all}\ a\in A, m\in M,
\end{equation}
where $a\star_1m$ and $m\star_1 a$ are the first terms in the deformations
of $M$ as left and right $A$-modules:
$$
a\star m=a\cdot m+a\star_1m\hbar+\cdots
\quad\mbox{and}\quad
m\star a=m\cdot a+m\star_1 a\hbar+\cdots.
$$
where $\star$ are the deformed (left and right) actions of $A_\hbar$ on $M_\hbar$.
\end{definition}

The following theorem about deformation quantization of Poisson
bimodules is proved by Chemla:

\begin{theorem}[Chemla \cite{Chemla} Corollary 21]\label{thm:Chemla}
Let $A$ be the Poisson algebra of a Poisson manifold,
and $M$ be a Poisson $A$-bimodule.
Then
$$
\mathrm{HP}^\bullet(A[\![\hbar]\!], M[\![\hbar]\!])\cong
\mathrm{HH}^\bullet(A_\hbar, M_\hbar).$$
\end{theorem}

We next apply this theorem to the case
of Poisson algebras
with nontrivial modular vector fields.
To this end, we first have to introduce the notion
of Artin-Schelter regular algebras.

\subsection{Artin-Schelter regular algebras}

Artin-Schelter regular algebras were introduced by Artin and Schelter
in \cite{AS}):

\begin{definition}[AS-regular algebra]
A connected graded $k$-algebra A is called {\it AS-regular}
of dimension $n$ if
\begin{enumerate}
\item[(1)] A has finite global dimension $n$, and

\item[(2)] A is Gorenstein, that is, $\mathrm{Ext}_A^i
(k,A)=0$ for $i\neq n$ and $\mathrm{Ext}_A^n
(k,A)\simeq k$.
\end{enumerate}
\end{definition}

In the literature, an AS-regular algebra is also called a {\it twisted
Calabi-Yau algebra}, due to the following.

\begin{theorem}[Reyes-Rogalski-Zhang \cite{RRZ} Lemma 1.2]
Suppose $A$ is as above. Then $A$ is AS-regular if and only
if it is twisted Calabi-Yau; that is, $A$ satisfies the following two conditions:
\begin{enumerate}
\item[$(1)$] A is homologically smooth, that is, $A$, viewed as an
$A^e$-module, has a bounded, finitely generated projective resolution;

\item[$(2)$] there exists an integer $n$ and an algebra automorphism $\sigma$ of $A$ such that
\[
\mathrm{Ext}^i_{A^e}(A, A^e)\cong
\left\{\begin{array}{cl} A_{\sigma},&\mbox{if}\; i=n,\\
	0,&\mbox{otherwise}
\end{array}\right.
\]
as $A^e$-modules.
\end{enumerate}
\end{theorem}

In the above theorem,
$A_\sigma$ is $A$ with the {\it twisted} $A$-bimodule structure given by
\[
a\cdot b\cdot c:=ab\sigma(c).
\]
for any $a,b,c\in A$, and $\sigma$ is usually called the {\it Nakayama automorphism} of $A$.
If $\sigma=Id$, then $A$ is called {\it Calabi-Yau} in the sense of Ginzburg \cite{Ginzburg}.

In 2008, Brown and Zhang obtained a refinement of Van den Bergh's
noncommutative Poincar\'e duality:

\begin{theorem}[\cite{BZ} Corollary 0.4]\label{thm:BZ}
Suppose $A$ is an AS-regular algebra of dimension $n$.
Then we have the following isomorphism
\[\mathrm{HH}^\bullet(A)\cong\mathrm{HH}_{n-\bullet}(A,A_\sigma),\]
where $\mathrm{HH}^\bullet(-)$ and $\mathrm{HH}_\bullet(-)$ are the Hochschild
cohomology and homology respectively.
\end{theorem}

\begin{example}\label{example:2dimAS}
Let $A= k\langle x_1,\cdots, x_n\rangle/(f)$,
where
$$
f=(x_1,\cdots, x_n)g(x_1,\cdots, x_n)^T, g\in \mathrm{GL}_n(k),
$$
and $(f)$ means the ideal generated by $f$.
Then $A$ is an AS-regular algebra.
Observe that $A$ is a graded algebra;
for $x=\sum k_i x_i$, $k_i\in k$,
let
$$
{\sigma}(x)=-(x_1,\cdots, x_n)g^Tg^{-1}(k_1,\cdots, k_n)^T,
$$
and extend it to the whole $A$. The $\sigma$ thus defined
is the Nakayama automorphism of $A$.
\end{example}

\begin{example}
[The Quantum affine space]
Let $Q=\left(\begin{array}{ccc}
q_{11}&\cdots & q_{1n}\\
\vdots&\ddots&\vdots\\
q_{n1}&\cdots & q_{nn}
\end{array}
\right)$ be an $n\times n$ matrix over $k$ with
$q_{ii}=1, q_{ij}q_{ji}=1$, for $1\leq i,j\leq n$.
Let $A=k\langle x_1,\cdots, x_n\rangle/(x_jx_i-q_{ij}x_ix_j)$.
Then $A$ is an AS-regular algebra with
the Nakayama automorphism $\sigma$ given by
$$
\sigma(x_i)=(\Pi_{j=1}^nq_{ji})x_i,\quad i=1,\cdots, n.
$$
\end{example}

\subsection{Quantization of the modular vector fields}

Now let $\nu$ be the modular vector field of a Poisson algebra $A$,
and $\nu_\hbar=\nu+\nu_1\hbar+\cdots$ be the modular vector field
with respect to $\pi_\hbar$. Since $\nu_\hbar$ is
a Poisson cocycle, its image under Kontsevich's $L_\infty$ map
gives a Hochschild cocycle, denoted by $\sigma$, which is in fact
$\exp(\nu_\hbar)$; see
Dolgushev \cite[Theorem 2]{Dolgushev} for a proof.

\begin{lemma}[Dolgushev]\label{lemma:Dolgu}
Let $A$ be as above.
Let $\nu$ be the modular vector field of $A$.
Then
$$(A_\nu)_\hbar=(A_\hbar)_\sigma,$$
up to an automorphism of $A[\![\hbar]\!]$ whose leading term is $Id$.
In other words, $\sigma$ is the deformation quantization of $\nu$ by
Kontsevich's $L_{\infty}$-quasi-isomorphism $\mathcal{U}$.
\end{lemma}

\begin{proof}
By the argument above, we only need to show that for any $a, m\in A$,
they satisfy \eqref{eq:infdeformationofmodules}. In fact,
\begin{eqnarray*}
a\star m-m\star\sigma(a)
&\equiv&\hbar(a\star_1 m -m\star_1a +m\cdot\nu(a))\\
&=&\hbar(\{a,m\}+m\cdot\nu(a))\\
&=&\hbar\{a,m\}_\nu\quad\mbox{mod}\; \hbar^2.
\end{eqnarray*}
The lemma now follows.\end{proof}

\begin{lemma}
Let $(A=\mathbb R[x_1,\cdots x_n],\pi)$ be a Poisson
algebra and $\nu$ be the corresponding modular vector field.
We have
\begin{equation}\label{eq:isooftwistedhomology}
\mathrm{HP}_\bullet(A[\![\hbar]\!], (A_\nu)[\![\hbar]\!])
\cong\mathrm{HH}_\bullet(A_\hbar, (A_\hbar)_\sigma).
\end{equation}
\end{lemma}

\begin{proof}
By above lemma, $(A_\nu)_\hbar=(A_\hbar)_\sigma$. The lemma
now follows Chemla's result Theorem \ref{thm:Chemla}.
\end{proof}

\begin{theorem}\label{thm:DQ+PDforpoly}
Suppose $A=\mathbb{R}[x_1,\ldots,x_n]$ is a Poisson algebra. Then
the diagram
\begin{equation}\label{diag:DQ+PDforpoly}
\begin{split}
\xymatrixcolsep{3pc}
\xymatrix{
\mathrm{HP}^\bullet(A[\![\hbar]\!])\ar[r]^-{\cong}\ar[d]^{\cong}
&\mathrm{HP}_{\bullet}(A[\![\hbar]\!],A_{\nu}[\![\hbar]\!])\ar[d]^{\cong}\\
\mathrm{HH}^{\bullet}(A_{\hbar})\ar[r]^-{\cong}&\mathrm{HH}_\bullet(A_{\hbar},(A_{\nu})_{\hbar})
}
\end{split}
\end{equation}	
commutes.
\end{theorem}

\begin{proof}
Dolgushev showed
that $A_\hbar$
satisfies the following
$$
\mathrm{Ext}^i_{A_\hbar\otimes_{\mathbb{R}[\![\hbar]\!]}
A_{\hbar}^{op}}(A_\hbar, A_\hbar\otimes
_{\mathbb{R}[\![\hbar]\!]} A_\hbar^{op})
\cong\left\{
\begin{array}{cl}
(A_\hbar)_\sigma,&i=n,\\
0,&\mbox{otherwise},
\end{array}
\right.
$$
where $\sigma$ is the deformation quantization of $\nu$ as in the previous two lemmas
(see Dolgushev \cite[Proposition 2]{Dolgushev}).
This means $A_\hbar$ is an AS-regular algebra over $\mathbb{R}[\![\hbar]\!]$
of dimension $n$, which then implies the noncommutative Poincar\'e duality
(see Theorem \ref{thm:BZ})
$$
\mathrm{HH}^\bullet(A_\hbar)\cong\mathrm{HH}_{n-\bullet}(A_\hbar, (A_\hbar)_\sigma).
$$
Combining \eqref{eq:twistedPD1}, \eqref{eq:isoofPoissonandHochschild}
and \eqref{eq:isooftwistedhomology} we get the isomorphisms in \eqref{diag:DQ+PDforpoly}.

Chemla proved in \cite[Theorem 10]{Chemla} that
for any Poisson $A$-bimodule $M$,
there is a quasi-isomorphism of $L_\infty$-modules
from the modules over $T_{\mathrm{poly}}(A)$
to modules over $D_{\mathrm{poly}}(A)$,
which she denotes by $T_{\mathrm{poly}}(M)$
and $D_{\mathrm{poly}}(M)$ respectively.
Such a quasi-isomorphism generalizes
the $L_\infty$-quasi-isomorphism from $T_{\mathrm{poly}}(A)$
to $D_{\mathrm{poly}}(A)$ of Kontsevich.
Then by a similar argument to that of Kontsevich
she gets the above Theorem \ref{thm:Chemla}, which is more precisely the
following commutative diagram
\begin{equation*}
\xymatrixcolsep{3pc}
\xymatrix{
\mathrm{HP}^\bullet(A[\![\hbar]\!])\ar@{~>}[r]\ar[d]^{\cong}
&\mathrm{HP}_{\bullet}(A[\![\hbar]\!],M[\![\hbar]\!])\ar[d]^{\cong}\\
\mathrm{HH}^{\bullet}(A_{\hbar})\ar@{~>}[r]&\mathrm{HH}_\bullet(A_{\hbar},M_{\hbar}),
}
\end{equation*}	
where the horizontal curved arrows mean the Lie algebra actions.
Restricting to the case where $M=A_\nu$, with the Poincar\'e duality
taken into account we get the commutativity of \eqref{diag:DQ+PDforpoly}.
\end{proof}

\subsubsection{Diagonalizable Nakayama automorphism}
Now we study the deformation quantization
of diagonalizable modular fields. For an associative
$k$-algebra $A$ and $\sigma\in\mathrm{Aut}(A)$,
recall that $\sigma$ is called {\it diagonalizable}
(or sometimes {\it semi-simple}) if there is a subset 
$\Sigma\subseteq k\backslash\{0\}$
and a decomposition of $k$-vector spaces
$$A=\bigoplus_{\lambda\in\Sigma}A_\lambda,\quad
A_\lambda=\{a\in A\,|\,\sigma(a)=\lambda a\}.$$ 
(See \cite[\S7.3]{KK} and also compare it with the Poisson
case given in \S\ref{subsect:diagonalizble}.)
The following is straightforward.

\begin{lemma}\label{lemma:DQofmodularsym}
Suppose $A$ is a Poisson algebra. Let $A_\hbar$ be its deformation
quantization.
If $\nu$ is diagonalizable, then $\sigma=\exp(\hbar\nu)$ is the diagonalizable
Nakayama automorphism of $A_\hbar$.
\end{lemma}

Kowalzig and Kr\"ahmer proved in \cite[Theorem 1.5]{KK}
that, {\it for an AS-regular algebra with diagonalizable Nakayama automorphism,
its Hochschild cohomology has a Batalin-Vilkovisky algebra structure,
whose Batalin-Vilkovisky operator generates the Gerstenhaber
bracket on the cohomology}. Thus in the light of Lemma \ref{lemma:DQofmodularsym},
combining this result with
Theorem \ref{thm:DQ+PDforpoly} we obtain the following.

\begin{theorem}\label{thm:DQ+BVforPoisson}
Suppose $A$ is a Poisson algebra. Let $A_\hbar$ be its deformation
quantization.
If $A$ has a diagonalizable modular vector field,
then we have an isomorphism
\begin{equation}
\mathrm{HP}^\bullet(A[\![\hbar]\!])\cong\mathrm{HH}^\bullet(A_\hbar)
\end{equation}
of Batalin-Vilkovisky algebras.
\end{theorem}

\section{Frobenius Poisson algebras}\label{sect:FrobPoisson}

In \cite{ZVOZ}, Zhu, Van Oystaeyen and Zhang introduced
the notion of Frobenius Poisson algebras, that
is, Poisson algebras with a non-degenerate pairing,
and studied the structures on their (co)homology.
In this subsection, we
study these algebras with diagonalizable
modular vector fields, and their twisted Poincar\'e duality,
Koszul duality and deformation quantization.

\subsection{Modular symmetry and Poincar\'e duality}

Let us start with the definition of Frobenius algebras.

\begin{definition}[Frobenius algebra]\label{def:Frobeniusalg}
A finite dimensional graded associative $k$-algebra
$A$ is called {\it Frobenius} of dimension $n$
if it is equipped with a bilinear, non-degenerate pairing of degree $n$
$$
\langle-,-\rangle: A \otimes A\to k
$$
such that
$
\langle a\cdot b, c\rangle=\langle a, b\cdot c\rangle
$, for all homogeneous $a, b, c\in A$.
\end{definition}

Suppose $A$ is a Frobenius algebra, then the nondegeneracy of the pairing
in the above definition is equivalent to saying that there is
an isomorphism
$$\eta: A\to A^*,\, a\mapsto \langle -, a\rangle$$
of left $A$-modules, but not necessarily an isomorphism of $A$-bimodules,
where $A^*:=\mathrm{Hom}(A, k)$. We shall discuss this more in \S\ref{subsect:DQofFPA}
below.

\begin{example}
Suppose $A^!=\mathbf{\Lambda}(\xi_1,\cdots,\xi_n)$ is the exterior algebra;
in what follows we view it as the graded symmetric algebra generated by
$\xi_1,\cdots,\xi_n$ with each grading $|\xi_i|=-1$.
There is a degree $n$ $A^!$-module isomorphism
\begin{eqnarray*}
\eta^!:A^{!}\rightarrow  A^{\ac},\;
\xi_{i_1}\cdots \xi_{i_p}\mapsto \eta^!(\xi_{i_1}\cdots \xi_{i_p})
\end{eqnarray*}
where
$$\eta^!(\xi_{i_1}\cdots \xi_{i_p}):=\sum_{s\in S_{p, n-p}}\langle\xi_{i_1}\cdots \xi_{i_p},
\xi^*_{s_1}\cdots \xi^*_{s_p}\rangle\cdot\xi^*_{s_{p+1}}\cdots \xi^*_{s_n},
$$
$A^{\ac}:=(A^{!})^*$, $\xi_i^*$'s are the linear duals of $\xi_i$'s, for $i=1,\cdots, n$,
and the sum runs over all $(p, n-p)$-shuffles $s$ of $(1,\cdots, n)$.
Recall that a $(p, n-p)$-{\it shuffle} is a permutation $s$ of $(1, \cdots, n)$
such that $s_1<\cdots< s_p, s_{p+1}<\cdots<s_n$.
It is direct to see that $\eta^!$
is non-degenerate and hence gives a Frobenius algebra structure on $A^!$.
We also write $\eta^!$ in the form
$\xi_1^*\cdots \xi_n^*$, and call it the volume form
of $A^!$.
\end{example}

\begin{definition}[Zhu-Van Oystaeyen-Zhang \cite{ZVOZ}]
A graded Poisson algebra $A$
is called {\it Frobenius Poisson}
if it is moreover a Frobenius algebra.
\end{definition}

For a Frobenius Poisson algebra, say $A$,
there is a differential calculus structure associated to it, which
is different to the one given in Example \ref{ex:differentialcalculusonPoisson}.
In fact, suppose $A$ is a Frobenius Poisson algebra.
Then
any $f\in\mathfrak X^{p}(A)$ and $\alpha\in\mathfrak X^{q}(A; A^*)$,
let
$
f\cap \alpha\in\mathfrak X^{p+q}(A; A^*)
$
be given by
\begin{equation}\label{def:innerproductondualvectors}
(f\cap\alpha)(a_1, \cdots, a_{p+q}):=
\sum_{s\in S_{p,q}}\mathrm{sgn}(\sigma)
f(a_{s_1},\cdots, a_{s_p})\cdot
\alpha(a_{s_{p+1}},\cdots, a_{s_{p+q}}),
\end{equation}
where $\sigma$ runs over all $(p,q)$-shuffles of $(1,\cdots, p+q)$.
Observing that
\begin{eqnarray*}
\mathfrak X^\bullet_A(A^*)&=&\mathrm{Hom}_A(\Omega^\bullet(A), A^*)\\
&=&\mathrm{Hom}_A\big(\Omega^\bullet(A), \mathrm{Hom}_k(A, k)\big)\\
&=&\mathrm{Hom}_k(A\otimes_A\Omega^\bullet(A) , k)\\
&=&\mathrm{Hom}_k(\Omega^\bullet(A), k).
\end{eqnarray*}
We dualize the de Rham differential $d$ on $\Omega^\bullet(A)$
and obtain a differential $d^*$ on $\mathrm{Hom}(\Omega^\bullet(A), k)$,
i.e., on $\mathfrak X^\bullet(A; A^*)$, which commutes with the Poisson coboundary
(see \cite[Theorem 4.10]{ZVOZ} for a proof).
The following proposition is obtained by Zhu-Van Oystaeyen-Zhang
(see \cite[\S3-4]{ZVOZ} for a complete proof).

\begin{proposition}Let $A$ be a Frobenius Poisson algebra. Then
$$
(\mathrm{HP}^\bullet (A), \mathrm{HP}^\bullet(A, A^*),\cup,\{-,-\}, \cap, d^*)
$$
form a differential calculus, where
$\cup$ and $\{-,-\}$ are as in the above example, and $\cap$ is $\iota$
given by \eqref{eq:defiota} and $d^*$ is the dual de Rham differential
given by \eqref{diag:unimodularPoisson222}.
\end{proposition}

In what follows we denote by $A^!$ a Frobenius Poisson algebra,
and by $A^{\ac}$ its linear dual.
From the nondegeneracy of the pairing
we in fact get an isomorphism
$\eta^!: A^!\to A^{\ac}$
which further induces an isomorphism of vector spaces
$$
\iota_{(-)}\eta^!:
\mathfrak{X}^\bullet_{A^!}(A^!)\to
\mathfrak{X}^\bullet_{A^!}(A^\ac)
$$
given by
\begin{equation}\label{eq:defiota}
\iota_{\varphi}\eta^!:=\{(a_1,\cdots, a_p)\mapsto \eta^!(\varphi(a_1,\cdots,a_p))\},
\quad\mbox{for}\; \varphi\in\mathfrak X^p(A^!),\; a_1,\cdots,a_p\in A^!.
\end{equation}
Again, $\iota_{(-)}\eta^!$ gives the following diagram
\begin{equation}\label{diag:unimodularFrobPoisson}
\begin{split}
\xymatrixcolsep{3pc}
\xymatrix{
\mathfrak{X}^\bullet_{A^!}(A^!)\ar[r]^-{\iota_{(-)}\eta^!}_{\cong}\ar[d]_{\delta}&
\mathfrak{X}^{\bullet}_{A^!}(A^\ac)\ar[d]^{\delta}\\
\mathfrak{X}^{\bullet+1}_{A^!}(A^!)\ar[r]^-{\iota_{(-)}\eta^!}_{\cong}&\mathfrak{X}^{\bullet+1}_{A^!}(A^\ac).
}
\end{split}
\end{equation}
of vector spaces, which
in general does not commute with the boundaries on each side,
since $\eta^!$ is not a Poisson cocycle.
To adjust this discrepancy, we do the same procedure
as in the Poisson algebra case.
Namely, let
$$
\mathrm{Div}: \mathfrak X_{A^!}^\bullet(A^!)\to\mathfrak X_{A^!}^{\bullet-1}(A^!)
$$
be such that the following diagram
\begin{equation}\label{diag:unimodularPoisson222}
\begin{split}
\xymatrixcolsep{3pc}
\xymatrix{
\mathfrak{X}_{A^!}^\bullet(A^!)\ar[r]^-{\iota_{(-)}\eta^!}_{\cong}\ar[d]_{\mathrm{Div}}
&\mathfrak X_{A^!}^{\bullet}(A^{\ac})\ar[d]_{d^*}\\
\mathfrak{X}_{A^!}^{\bullet-1}(A^!)\ar[r]^-{\iota_{(-)}
\eta^!}_{\cong}&\mathfrak X_{A^!}^{\bullet-1}(A^{\ac}),
}
\end{split}
\end{equation}
commutes, where $d^*$ is the dual of the de Rham differential.
Let $\nu^!=-\mathrm{Div}(\pi^!)$, which is also called
the {\it modular vector field} of $A^!$. Analogously to Lemma \ref{TwistedPoincare},
for any $\varphi\in \mathfrak{X}^p(A^!)$, we have
\begin{equation}\label{formula:XuinFrobcase}
\partial(\varphi\cap\eta^!)+\nu^!(\varphi\cap\eta^!)=\delta(\varphi)\cap\eta^!.
\end{equation}
Combining \eqref{diag:unimodularFrobPoisson}
and \eqref{formula:XuinFrobcase}, with the appropriate
degree on the cohomology
taken into account,
yields the following.

\begin{theorem}[\cite{LWW2} \S3.1]\label{thm:twistedPD2}
Let $A^!$ be a Poisson exterior algebra, and $\nu^!$
be the corresponding modular vector field. Then
\begin{equation}\label{eq:twistedPD2}
\mathrm{HP}^\bullet(A^!)\cong\mathrm{HP}^{\bullet-n}(A^!,A^{\ac}_{\nu^!}).
\end{equation}
\end{theorem}

Now let us move to the Frobenius Poisson algebra case.
Let $A^!$ be a Frobenius Poisson algebra with diagonalizable modular vector field $\nu^!$.
The following three statements are completely parallel to
Lemma \ref{lemma:homotopy}--Theorem \ref{thm:BVforPoisson},
and we leave their proofs to the interested reader.

\begin{lemma}
On the Poisson cochain complex $\mathrm{CP}^\bullet(A^!,A^\ac_{\nu^!})$, we have
$$
\partial_{\nu^!}\circ d^*+d^*\circ\partial_{\nu^!}=\tilde \nu^!.
$$
\end{lemma}

\begin{corollary}
$(\mathrm{HP}^\bullet(A^!),\mathrm{HP}^\bullet(A^!,A^\ac_{\nu^{!}}), \cup, \{-,-\},\iota, d^*)$
forms a differential calculus with duality.
\end{corollary}

\begin{theorem}[See also \cite{WWZZ}]\label{thm:BVforFrobPoisson}
Suppose $A^!$ is a Frobenius
Poisson algebra with diagonalizable
modular vector field $\nu^!$, then
$\mathrm{HP}^\bullet(A^!)$ has a Batalin-Vilkovisky
algebra structure whose Batalin-Vilkovisky operator generates
the Schouten bracket.
\end{theorem}

\subsection{Koszul duality for Poisson algebras}

From now on, we focus on {\it quadratic} Poisson algebras.
As we mentioned before, Shoikhet showed that the Koszul
dual of a quadratic Poisson polynomial algebra is again quadratic
Poisson.
In this section, we study the modular symmetry under Koszul duality,
and the main result is Theorem \ref{thm:Koszul+PDforPoisson}.

\begin{definition}
Let $A=\mathbb R[x_1,\cdots, x_n]$ be the real polynomial algebra in $n$ variables.
A Poisson structure on $A$, say $\pi$, is called {\it quadratic} if it is of the form
\begin{equation}\label{formula:quadratic_Poisson}
\pi=\sum_{i_i,i_2,j_1,j_2}c_{i_1i_2}^{j_1j_2}x_{i_1}x_{i_2}
\frac{\partial}{\partial x_{j_1}}\wedge\frac{\partial}{\partial x_{j_2}},
\quad c_{i_1i_2}^{j_1j_2}\in\mathbb R.
\end{equation}
\end{definition}

\begin{definition}\label{def:quadraticdualPoisson}
If $A=\mathbb R[x_1,\cdots, x_n]$ is the polynomial algebra
with a quadratic bivector
$$\pi=\sum_{i_i,i_2,j_1,j_2}c_{i_1i_2}^{j_1j_2}x_{i_1}x_{i_2}
\frac{\partial}{\partial x_{j_1}}\wedge\frac{\partial}{\partial x_{j_2}},$$
then its {\it Koszul dual}, denoted by $A^!$, is the graded symmetric
algebra
$$A^!=\mathbf\Lambda(\xi_1,\cdots,\xi_n),\quad |\xi_i|=-1, i=1,\cdots, n$$
with the dual bivector
\begin{equation}\label{formula:quadratic_dualPoisson}
\pi^!=\sum_{i_1,i_2,j_1,j_2} c_{i_1i_2}^{j_1j_2}\xi_{j_1}\xi_{j_2}\frac{\partial}{\partial
\xi_{i_1}}\frac{\partial}{\partial \xi_{i_2}}.
\end{equation}
Under the correspondence
\begin{equation}\label{eq:correspond}
x_i\leftrightarrow \frac{\partial}{\partial \xi_i}\quad\mbox{and}\quad
\frac{\partial}{\partial x_i}\leftrightarrow \xi_i
\end{equation}
between
the sets of polyvectors on $A$ and on $A^!$,
it is direct to check that $\pi$ is Poisson if and only if $\pi^!$
is Poisson. We call $(A^!,\pi^!)$ the {\it Koszul dual Poisson algebra} of $(A,\pi)$.
\end{definition}

\begin{proposition}[See also \cite{CCEY}]\label{prop:KoszulforPoisson}
Let $(A, \pi)$ and $(A^!, \pi^!)$ be the quadratic Poisson algebras
Koszul dual to each other as given in Definition
\ref{def:quadraticdualPoisson}. Then we have isomorphisms
$$
\mathrm{HP}^\bullet(A)\cong\mathrm{HP}^\bullet(A^!)\quad
\mbox{and}\quad
\mathrm{HP}_\bullet(A)\cong\mathrm{HP}^{-\bullet}(A^!, A^{\ac}),
$$
where $A^{\ac}$ is $(A^!)^*=\mathrm{Hom}(A^!, k)$.
\end{proposition}

\begin{proof}
Since $A=\mathbb R[x_1,\cdots,x_n]$, we have
\begin{eqnarray}\label{Poissoncochainforpolynomials}
\mathfrak X^\bullet(A)
=\mathbf\Lambda\Big(x_1,\cdots,x_n,
\frac{\partial}{\partial x_1},\cdots,\frac{\partial}{\partial x_n}\Big)
\end{eqnarray}
and similarly,
\begin{eqnarray}\label{Poissoncochainforexteriorpolynomials}
\mathfrak X^\bullet(A^!)&=&
\mathbf{\Lambda}
\Big(\xi_1,\cdots,\xi_n,
\frac{\partial}{\partial \xi_1},\cdots,\frac{\partial}{\partial \xi_n}\Big),
\end{eqnarray}
where the gradings are given as follows:
$$
|x_i|=0, \left|\frac{\partial}{\partial x_i}\right|=-1,
|\xi_i|=-1,\left|\frac{\partial}{\partial \xi_i}\right|=0, \quad
i=1,\cdots, n.$$
Under the correspondence \eqref{eq:correspond}
we obtain an isomorphism of chain complexes
$$
\Psi:\mathrm{CP}^\bullet(A)\cong\mathrm{CP}^\bullet(A^!),
$$
which gives the first isomorphism on the cohomology.

For the second isomorphism, let us notice that
\begin{equation}
\Omega^\bullet(A)=\mathbf\Lambda(x_1,\cdots, x_n, dx_1,\cdots, dx_n)\label{formula:Poissonchaincpx},
\end{equation}
and
$$
\Omega^\bullet(A^!)=\mathbf\Lambda(\xi_1,\cdots,\xi_n,d\xi_1,\cdots, d\xi_n),
$$
where
$|dx_i|=1,|d\xi_i|=0$ for $i=1,\cdots, n$.
We therefore have
\begin{eqnarray}
\mathfrak X^\bullet_{A^!}(A^{\ac})
=\mathrm{Hom}_{A^!}(\Omega^\bullet(A^!), A^{\ac})
\cong\mathrm{Hom}_k(\Omega^\bullet(A^!), k)=\mathbf\Lambda\Big(
\frac{\partial}{\partial \xi_1},\cdots,
\frac{\partial}{\partial\xi_n},
\xi_1^*,\cdots,\xi_n^*\Big),\label{formula:coPoissoncochaincpx}
\end{eqnarray}
where $|\xi_i^*|=1$.
Thus under the correspondence \eqref{eq:correspond}
together with $dx_i\leftrightarrow \xi_i^*$
we get a canonical grading preserving isomorphism of vector spaces:
$$
\Phi:\Omega^\bullet(A)\to\mathfrak X_{A^!}^{\bullet}(A^{\ac}),
\;
x_i\mapsto\frac{\partial}{\partial \xi_i},\;
dx_i\mapsto\xi_i^*.
$$
It is a direct check that $\Phi$ is a chain map, and
thus we obtain an
isomorphism of Poisson complexes which then induces an isomorphism on homology
$\mathrm{HP}_\bullet(A)\cong\mathrm{HP}^{-\bullet}(A^!, A^{\ac}).$
\end{proof}

\subsubsection{Koszul duality and modular symmetry}

We now study the behavior of the modular vector field under Koszul duality.

\begin{proposition}\label{prop:isoofmodular}
Suppose $A=(\mathbb R[x_1,\cdots,x_n],\pi)$
and $A^!=(\Lambda(\xi_1,\cdots,\xi_n),\pi^!)$ are Koszul dual Poisson
algebras.
Then under the correspondence \eqref{eq:correspond}
the modular vector field
$\nu$ of $A$ corresponds to $\nu^!$ of $A^!$.
\end{proposition}

\begin{proof}
It is direct to check that the modular vector field \begin{eqnarray*}
 \nu &=&-\mathrm{Div}(\pi)\\
&=&\sum_{1\leq i<l\leq n}c^{il}_{ij}x_j\dfrac{\partial}{\partial x_l}+\sum_{1\leq j<l\leq n}c^{jl}_{ij}x_i\dfrac{\partial}{\partial x_l}
+\sum_{1\leq k<i\leq n}c^{ki}_{ij}x_j\dfrac{\partial}{\partial x_k}+\sum_{1\leq k<j\leq n}c^{kj}_{ij}x_i\dfrac{\partial}{\partial x_k} .
\end{eqnarray*}
On the other hand, we have
\begin{eqnarray*}
\nu^!&=& -\mathrm{Div}(\pi^!)\\
&=&\sum_{1\leq i<l\leq n}c^{il}_{ij}\xi_l\dfrac{\partial}{\partial\xi_j}
+\sum_{1\leq j<l\leq n}c^{jl}_{ij}\xi_l\dfrac{\partial}{\partial\xi_i}
+\sum_{1\leq k<i\leq n}c^{ki}_{ij}\xi_k\dfrac{\partial}{\partial\xi_j}+\sum_{1\leq k<j\leq n}c^{kj}_{ij}\xi_k\dfrac{\partial}{\partial\xi_i}.
\end{eqnarray*}
 Under the identification
\eqref{eq:correspond}
these two modular derivations are isomorphic to each other.
\end{proof}

From the above computation of $\nu$ we also get the following byproduct.

\begin{proposition}
Suppose $A=k[x_1,\cdots,x_n]$ is a Poisson algebra with Poisson structure $\pi$.
Take the volume form to be $\eta=dx_1\wedge dx_2\wedge\cdots\wedge dx_n$.
Then the modular vector field is diagonalizable
 if and only if $\pi$
 is of the form
 $$\pi=\sum_{i,j}c_{ij}x_ix_j\frac{\partial}{\partial x_{i}}\wedge\frac{\partial}{\partial x_{j}}.$$
\end{proposition}

Also as a corollary of Proposition \ref{prop:isoofmodular}, we obtain the following:

\begin{theorem}\label{thm:Koszul+PDforPoisson}
Let $A=\mathbb R[x_1,\cdots, x_n]$ be a quadratic Poisson
algebra, and let $A^!=\Lambda(\xi_1,\cdots,\xi_n)$
be its Koszul dual. Denote by $\nu$ and $\nu^!$
the modular vector fields of $A$ and $A^!$ respectively.
Then the following diagram
\begin{equation}\label{diag:Koszul+PDforPoisson}
\begin{split}
\xymatrixcolsep{3pc}
\xymatrix{
\mathrm{HP}^\bullet(A)\ar[r]^-{\cong} \ar[d]^{\cong}
&\mathrm{HP}_{n-\bullet}(A, A_\nu)\ar[d]^{\cong}\\
\mathrm{HP}^\bullet(A^!)\ar[r]^-{\cong} &\mathrm{HP}^{\bullet-n}(A^{!}, A^{\ac}_{\nu^!})
}
\end{split}
\end{equation}	
commutes.
\end{theorem}

\begin{proof}
With the results in Proposition \ref{prop:KoszulforPoisson}
and Theorems \ref{thm:twistedPD1} and \ref{thm:twistedPD2},
the only
thing that we need to prove is
$$
\mathrm{HP}_\bullet(A, A_\nu)\cong\mathrm{HP}^{-\bullet}(A^!, A^{\ac}_{\nu^!}).
$$
The proof of the second isomorphism in Proposition
\ref{prop:KoszulforPoisson} shows that
$$\mathrm{CP}_\bullet(A, A_\nu)\cong\mathrm{CP}^{-\bullet}(A^!, A^{\ac}_{\nu^!})$$
as chain complexes with respect to the Poisson boundary maps;
now Proposition \ref{prop:isoofmodular}
says that the twistings on both sides of the above complexes are also identical.
Taking the corresponding homology we get the commutative diagram \eqref{diag:Koszul+PDforPoisson},
and the theorem follows.
\end{proof}

\subsection{Deformation quantization of Frobenius Poisson algebras}\label{subsect:DQofFPA}
We now show that the deformation quantizations of Frobenius Poisson algebras
are Frobenius algebras.

In Definition \ref{def:Frobeniusalg} of a Frobenius algebra,
since the pairing is non-degenerate, there exists an automorphism
$\sigma^!$ such that $\langle ab, c\rangle = (-1)^{\vert c\vert(\vert a\vert+\vert b\vert)}
\langle\sigma^!(c)a, b\rangle$. Such a $\sigma^!$ is also called the
{\it Nakayama automorphism}
of $A^!$. The non-degeneracy of the pairing given by is equivalent to saying that
\begin{eqnarray*}
\eta^! :A^!\rightarrow  A^{\ac}_{\sigma^!}[-n], \,
	a\mapsto \langle-,a\rangle
\end{eqnarray*}
is an isomorphism of $A^!$-bimodules.
In 2016, Lambre, Zhou and Zimmermann obtained the following ``noncommutative Poincar\'e duality":

\begin{theorem}[\cite{LZZ} Proposition 3.3]\label{thm:LZZ}
Let $A^!$ be a Frobenius algebra of degree n with Nakayama
automorphism $\sigma^!$. Then there is an isomorphism
\[\mathrm{HH}^\bullet(A^!)\rightarrow \mathrm{HH}^{\bullet-n}(A^!,A^{\ac}_{\sigma^!}).\]
\end{theorem}

\subsubsection{Deformation quantization}
First, we recall that for graded Poisson algebras over supermanifolds,
Kontsevich's deformation quantization remains valid (see Cattaneo
and Felder \cite[Theorem 4.6]{CF} for a proof).

Now by the same argument as in the polynomial case,
$\nu^!$ can be deformation quantized via the Kontsevich map.
Denote by $\sigma^!$ its deformation quantization; then we have
(see Lemma \ref{lemma:Dolgu})
\begin{equation}\label{eq:quantizationofnu}
(A^{\ac}_{\nu^!})_\hbar\cong(A^{\ac}_{\hbar})_{\sigma^!}.
\end{equation}
This implies the following lemma.

\begin{lemma}\label{lemma:NakaofFrobPoisson}
Let $A^!$ and $\sigma^!$ be as above. Then $\sigma^!$ is the Nakayama automorphism
of $A^!$.
\end{lemma}

\begin{proof}
We have to show that for any $a, b\in A^!_\hbar$,
$$
\langle a, b\rangle =(-1)^{|a||b|}\langle \sigma^!(b), a\rangle.
$$
This is equivalent to showing that
$
A^!_\hbar\cong (A^{\ac}_\hbar)_{\sigma^!}
$
as $A_\hbar$-bimodules.

In fact,
$
A^!\cong A^{\ac}_{\nu^!}
$
as Poisson $A$-modules,
and therefore they have isomorphic deformation quantization. This implies
that
$$
A^!_\hbar\cong (A^{\ac}_{\nu^!})_\hbar
$$
as $A_\hbar$-bimodules. Combining it with \eqref{eq:quantizationofnu},
we get the lemma.
\end{proof}

Similarly to Theorem \ref{thm:DQ+PDforpoly}, we have the following.

\begin{theorem}\label{thm:DQ+PDforFrob}
Suppose $A^!$ is a Frobenius
Poisson algebra.
Then the diagram
\begin{equation}\label{diag:DQ+PDforFrob}
\begin{split}
\xymatrixcolsep{3pc}
\xymatrix{
\mathrm{HP}^\bullet(A^![\![\hbar]\!])\ar[r]^-{\cong}\ar[d]^{\cong}
&\mathrm{HP}^{\bullet-n}(A^![\![\hbar]\!],A^{\ac}_{\nu^!}[\![\hbar]\!])\ar[d]^{\cong}\\
\mathrm{HH}^{\bullet}(A^!_{\hbar})\ar[r]^-{\cong}&
\mathrm{HH}^{\bullet-n}(A^!_{\hbar},(A^{\ac}_\hbar)_{\sigma^!})
}
\end{split}
\end{equation}	
commutes.\end{theorem}

\begin{proof}
Observe that
the left vertical isomorphism is Kontsevich's
isomorphism \eqref{eq:isoofPoissonandHochschild},
the top horizontal isomorphism is
given by \eqref{eq:twistedPD2},
and the bottom horizontal isomorphism is
the right vertical isomorphism of \eqref{eq:HHtoHHtwisted}
with the Nakayama automorphism given by Lemma \ref{lemma:NakaofFrobPoisson}.

We now need to
prove
\begin{equation}\label{eq:HPtoHHtwisted}
\mathrm{HP}^{\bullet-n}(A^![\![\hbar]\!],A^{\ac}_{\nu^!}[\![\hbar]\!])
\cong
\mathrm{HH}^{\bullet-n}(A^!_{\hbar},(A^{\ac}_\hbar)_{\sigma^!}).
\end{equation}
In fact this follows from combining \eqref{eq:quantizationofnu}
and Theorem \ref{thm:Chemla}.
\end{proof}

\subsubsection{Diagonalizable Nakayama automorphism}
Analogously to Kowalzig and Kr\"ahmer \cite{KK},
for a Frobenius algebra with diagonalizable Nakayama automorphism,
Lambre, Zhou and Zimmermann proved in \cite[Theorem 4.1]{LZZ} that
its Hochschild cohomology also admits a Batalin-Vilkovisky
algebra structure, whose Batalin-Vilkovisky operator generates
the Gerstenhaber bracket on the cohomology.
Parallel to Theorem \ref{thm:DQ+BVforPoisson}, we have the following.

\begin{theorem}\label{thm:DQ+BVforFrobPoisson}
Suppose $A^!$ is a Frobenius
Poisson algebra. Let $A^!_\hbar$ be its deformation
quantization.
If $A^!$ has diagonalizable modular symmetry,
then we have an isomorphism
\begin{equation}
\mathrm{HP}^\bullet(A^![\![\hbar]\!])\cong\mathrm{HH}^\bullet(A^!_\hbar)
\end{equation}
of Batalin-Vilkovisky algebras.
\end{theorem}

\section{Poincar\'e duality, Koszul duality and deformation quantization}\label{sect:PDKDDQ}

In this section we study the deformation quantization
of quadratic Poisson algebras, which relates
the theorems obtained in previous sections.

\subsection{Koszul duality of AS-regular algebras}
We start with the Koszul duality theory for associative algebras.

Let $V$ be a finite-dimensional (possibly graded) vector space over $k$.
Denote by $TV$ the free algebra generated by $V$ over $k$;
that is, $TV$ is the tensor algebra generated by $V$.
Suppose $R$ is a subspace of $V\otimes V$, and let
$(R)$ be the two-sided ideal generated by $R$ in $TV$,
then the quotient algebra
$A:= TV/(R)$
is called
a {\it quadratic algebra}.
Let $A^!$ be the quadratic dual algebra of $A$; that is,
$A^!=TV^*/(R^{\perp})$, where
$R^{\perp}=\{r^*\in V^*\otimes V^*|r^*(R)=0\}$.
Let $A^{\ac}$ be the linear dual of $A^!$, called the quadratic dual coalgebra of $A$.
Choose a set of basis $\{e_i\}$ for $V$, and let $\{e_i^*\}$ be their duals in $V^*$.
There is a natural chain complex associated to $A$, called the {\it Koszul complex}:
\begin{equation}\label{Koszul_complex}
\xymatrix{
\cdots\ar[r]^-{\delta}&
A\otimes A^{\ac}_{i+1}\ar[r]^-{\delta}&
A\otimes A^{\ac}_{i}\ar[r]^-{\delta}&
\cdots\ar[r]&
A\otimes A^{\ac}_0\ar[r]^-{\delta}& k,
}
\end{equation}
where for any $r\otimes f\in A\otimes A^{\ac}$,
$\delta(r\otimes f)=\displaystyle\sum_i e_ir\otimes e_i^*f$.

\begin{definition}[Koszul algebra]
A quadratic algebra $A=TV/(R)$ is called Koszul
if the Koszul chain complex \eqref{Koszul_complex} is acyclic.
\end{definition}

A typical example of Koszul algebras
what we use throughout the paper is
the polynomial algebra $A=k[x_1,\cdots, x_n]$, whose Koszul dual is the exterior
algebra $A^!=\Lambda(\xi_1,\cdots,\xi_n)$.

We have the following Koszul duality for AS-regular and Frobenius algebras;
see Smith \cite[Proposition 5.10]{Smith} and Van den Bergh \cite[pp. 667]{VdB96}
for a proof.

\begin{proposition}\label{prop:ASregularFrob}
Let $A$ be a Koszul algebra, and let $A^!$ be its
Koszul dual algebra. Then $A$ is an AS-regular algebra
if and only if $A^!$ is Frobenius. Under this correspondence,
the Nakayama automorphism of $A$ is Koszul dual to
the Nakayama automorphism of $A^!$.
\end{proposition}

Combining Theorems \ref{thm:BZ} and \ref{thm:LZZ}
and Proposition \ref{prop:ASregularFrob},
the second author was able to prove the following.

\begin{theorem} [\cite{Liu} Lemma 5.8]\label{thm:Koszul+PDforassoc}
Let $A$ be a Koszul AS-regular algebra.
Let $A^!$ and $A^{\ac}$ be its Koszul dual algebra and coalgebra respectively.
Then the Nakayama automorphism $\sigma$ of $A$ is mapped to the Nakayama automorphism
$\sigma^!$ of $A^!$ under Koszul duality, and the following diagram
\begin{equation}\label{eq:HHtoHHtwisted}
\begin{split}
\xymatrix{
\xymatrixcolsep{3pc}
\mathrm{HH}^\bullet(A)\ar[d]^-{\cong}\ar[r]^-{\cong}&\mathrm{HH}_{n-\bullet}(A,A_\sigma)\ar[d]^{\cong}\\
\mathrm{HH}^\bullet(A^!)\ar[r]^-{\cong}&\mathrm{HH}^{\bullet-n}(A^!,A^{\ac}_{\sigma^!})
}
\end{split}
\end{equation}
commutes.
Moreover, if the Nakayama on $A$ and hence on $A^!$ is
diagonalizable, then
$$
\mathrm{HH}^\bullet(A)\cong\mathrm{HH}^\bullet(A^!)
$$
as Batalin-Vilkovisky algebras, whose
Batalin-Vilkovisky operators generate the
Gerstenhaber brackets on both sides.
\end{theorem}

\subsection{Koszul duality and deformation quantization}\label{sect:DQandKD}

One of the motivations of the current paper is the result of Shoikhet et al.
on the Koszul duality between the deformation quantizations
of quadratic Poisson polynomial algebras and their Koszul dual,
which is stated as follows
(see Shoikhet \cite[Theorem 0.3]{Shoikhet1}
and Calaque et al. \cite[ Theorem 8.6]{Calaque+} for a proof):
{\it Let $A=\mathbb R[x_1,\cdots,x_n]$ and $A^!$ its Koszul dual.
Then Kontsevich's deformation quantization of $A$ and $A^!$, denoted by $A_\hbar$
and $A^{!}_\hbar$ respectively, are also Koszul dual
to each other
as associative algebras over $\mathbb R[\![\hbar]\!]$.
}

Notice that
by Shoikhet \cite{Shoikhet}, the Koszul duality theory remain valid if $\mathbb R$ is replaced by
$\mathbb R[\![\hbar]\!]$, and therefore, the Koszul duality between $A_\hbar$
and $A^!_\hbar$ over $\mathbb R[\![\hbar]\!]$ in the above
theorem makes sense.
The following theorem is obtained in \cite[Theorem 1.5]{CCEY}: {\it
Let $A[\![\hbar]\!]$ and $A^![\![\hbar]\!]$
be Koszul dual Poisson algebras.
Then
we have the following commutative diagram of isomorphisms
\begin{equation}\label{diag:Koszul+DQ+Hochschildcohomology}
\begin{split}
\xymatrixcolsep{3pc}
\xymatrix{
\mathrm{HP}^\bullet(A[\![\hbar]\!])\ar[r]^-{\cong}\ar[d]^-{\cong}&\mathrm{HP}^\bullet(A^![\![\hbar]\!])\ar[d]^-{\cong}\\
\mathrm{HH}^\bullet(A_\hbar)\ar[r]^-{\cong}&\mathrm{HH}^\bullet(A^!_\hbar).
}
\end{split}
\end{equation}
}

For the twisted Poisson homology and the twisted Hochschild homology, we have the following.

\begin{theorem}\label{thm:Koszul+DQhomology}
Let $A[\![\hbar]\!]$ and $A^{\ac}[\![\hbar]\!]$
be Koszul dual Poisson algebras.
Then we have the following commutative diagram of isomorphisms
\begin{equation}\label{diag:Koszul+DQ+Hochschildhomology}
\begin{split}
\xymatrixcolsep{3pc}
\xymatrix{
\mathrm{HP}_\bullet(A[\![\hbar]\!],A_\nu[\![\hbar]\!])\ar[r]^-{\cong}\ar[d]^-{\cong}
&\mathrm{HP}^{-\bullet}(A^![\![\hbar]\!],A^{\ac}_{\nu^!}[\![\hbar]\!])\ar[d]^-{\cong}\\
\mathrm{HH}_\bullet(A_\hbar,(A_{\nu})_{\hbar})\ar[r]^-{\cong}
&\mathrm{HH}^{-\bullet}(A^!_\hbar,(A^{\ac}_{\nu^!})_{\hbar}).
}
\end{split}
\end{equation}
\end{theorem}

\begin{proof}
The top horizontal isomorphism is the right vertical isomorphism of
\eqref{diag:Koszul+PDforPoisson}; the bottom horizontal isomorphism
is the right vertical isomorphism of
\eqref{eq:HHtoHHtwisted};
the left vertical isomorphism is
\eqref{eq:isooftwistedhomology}; and the right vertical isomorphism
is \eqref{eq:HPtoHHtwisted}. The commutativity
follows from the Hochschild-Kostant-Rosenberg theorem.
\end{proof}

The following theorem summarizes the above several results (the {\it unimodular} case
has previously been studied in \cite{CCEY}).

\begin{theorem}\label{mainthm2}
Let $A[\![\hbar]\!]$ and $A^![\![\hbar]\!]$
be Koszul dual Poisson algebras.
Then the following diagram of isomorphisms
\begin{equation}\label{maindiagram}
\begin{split}
\xymatrixrowsep{1pc}
\xymatrixcolsep{1pc}
\xymatrix{
&\mathrm{HP}^\bullet(A^![\![\hbar]\!])\ar[rr]^-{\cong} \ar'[d][dd]^-{\cong}
&&\mathrm{HP}^{\bullet-n}(A^![\![\hbar]\!], A^{\ac}_{\nu^!}[\![\hbar]\!])\ar[dd]^-{\cong}\\
\mathrm{HP}^\bullet(A[\![\hbar]\!])\ar[rr]^-{\cong} \ar[dd]^-{\cong} \ar[ur]^-{\cong}
&&\mathrm{HP}_{n-\bullet}(A[\![\hbar]\!], A_\nu[\![\hbar]\!])\ar[dd]^-{\cong}\ar[ur]^-{\cong} &\\
&\mathrm{HH}^\bullet(A^!_{\hbar})\ar'[r][rr]^-{\cong}&&\mathrm{HH}^{\bullet-n}(A^!_{\hbar}, A^{\ac}_{\hbar,\sigma^!})\\
\mathrm{HH}^\bullet(A_{\hbar})\ar[rr]^-{\cong}\ar[ur]^-{\cong} &&\mathrm{HH}_{n-\bullet}(A_{\hbar}, A_{\hbar,\sigma})\ar[ur]^-{\cong}&
}
\end{split}
\end{equation}
commutes,
where the horizontal arrows are the Poincar\'e duality, the vertical arrows are given by deformation quantization,
and the slanted arrows are given by Koszul duality.
\end{theorem}

\begin{proof}
The top square of the diagram is given by \eqref{diag:Koszul+PDforPoisson},
the bottom square is given by \eqref{eq:HHtoHHtwisted},
the front square is given by \eqref{diag:DQ+PDforpoly},
the back square is given by \eqref{diag:DQ+PDforFrob},
and
the left and the right squares are given by \eqref{diag:Koszul+DQ+Hochschildcohomology}
and \eqref{diag:Koszul+DQ+Hochschildhomology} respectively.
\end{proof}

\subsection{Isomorphisms of Batalin-Vilkovisky algebras}

We continue to show that, for a quadratic
Poisson algebra with diagonalizable modular symmetry,
the left side
diagram in \eqref{maindiagram}
is a commutative diagram of isomorphisms of Batalin-Vilkovisky
algebras (see Theorem \ref{thm:BViso}). It induces a commutative diagram of isomorphisms
of gravity algebras on the corresponding negative cyclic homology (see Theorem \ref{thm:gravityiso}).

\begin{lemma}\label{lemma:Koszul+BVPoisson}
Let $A$ be a quadratic Poisson algebra.
Let $A^!$ be its Koszul dual algebra.
If the modular vector field $\sigma$ is diagonalizable, then
so is its Koszul dual $\sigma^!$.
In this case, we have
$$
\mathrm{HP}^\bullet(A)\cong\mathrm{HP}^\bullet(A^!)
$$
as Batalin-Vilkovisky algebras.
\end{lemma}

\begin{proof}
The first half follows from Proposition \ref{prop:isoofmodular}.
The Batalin-Vilkovisky algebra isomorphism follows from
Theorem \ref{thm:Koszul+PDforPoisson}.
\end{proof}

The following result is proved in \cite[Theorem 1.1]{Liu}:

\begin{lemma}\label{lemma:Koszul+BValg}
Let $A$ be a Koszul AS-regular algebra.
Let $A^!$ be its Koszul dual algebra.
If the Nakayama automorphism $\sigma$ is diagonalizable, then
so is its Koszul dual $\sigma^!$.
In this case, we have
$$
\mathrm{HH}^\bullet(A)\cong\mathrm{HH}^\bullet(A^!)
$$
as Batalin-Vilkovisky algebras.
\end{lemma}

We now reach to the proof of the following two theorems, which supersede
the results obtained in \cite{CCEY,CEL} for unimodular Poisson algebras.

\begin{theorem}\label{thm:BViso}
Suppose $A=\mathbb R[x_1,\cdots, x_n]$.
For a quadratic Poisson structure on $A[\![\hbar]\!]$ with diagonalizable
modular vector field, the following
\begin{equation}\label{diag:commBV}
\begin{split}
\xymatrixcolsep{3pc}
\xymatrix
{
\mathrm{HP}^\bullet(A[\![\hbar]\!])\ar[r]^-{\cong}\ar[d]^-{\cong}&\mathrm{HP}^\bullet(A^![\![\hbar]\!])\ar[d]^-{\cong}\\
\mathrm{HH}^\bullet(A_\hbar)\ar[r]^-{\cong}&\mathrm{HH}^\bullet(A^!_\hbar)
}
\end{split}
\end{equation}
is a commutative diagram of isomorphisms of Batalin-Vilkovisky algebras.
\end{theorem}

\begin{proof}
Combine the left side diagram of
\eqref{maindiagram}
with the first halves of Theorems \ref{thm:DQ+BVforPoisson}
and \ref{thm:DQ+BVforFrobPoisson} and Lemmas
\ref{lemma:Koszul+BVPoisson} and \ref{lemma:Koszul+BValg}.
\end{proof}

\subsection{The gravity algebra structure}\label{subsect:gravity}

In this last subsection, we briefly discuss the gravity algebra
structure on the negative cyclic homology of Poisson algebras
with diagonalizable modular vector field.

The notion of gravity algebra
was introduced by Getzler in
\cite{Getzler2}; it
plays an important role
in the study of equivariant topological conformal field theory.
In \cite{CEL}, the first two authors of the current paper together with Eshmatov
showed that the negative
cyclic Poisson homology of a {\it unimodular} Poisson algebra
has a gravity algebra structure.
In what follows we generalize the result of \cite{CEL}
to the case of Poisson algebras whose modular vector field is diagonalizable.

\begin{definition}[Getzler \cite{Getzler2}]
Suppose $V$ is a (graded) vector space over $k$.
A {\it gravity algebra structure} on $V$ consists of a sequence of
(graded) skew symmetric operators
(called the {\it higher Lie brackets})
\[
\{x_1,\ldots, x_n\}: V^{\otimes n}\rightarrow V, n=2,3,\cdots
\]
such that
\begin{eqnarray*}
\sum_{1\leq i<j\leq n}(-1)^{\epsilon_{ij}}\{\{x_i,x_j\},
x_1\ldots,\widehat{x_i},\ldots,\widehat{x_j},\ldots, x_n,y_1,\ldots, y_m\}\\
=\left\{\begin{array}{cl}
\{\{x_1\ldots, x_n\},y_1,\ldots, y_m\}, &\mbox{if}\; m>0,\\
0,&\mbox{otherwise,}
\end{array}\right.
\end{eqnarray*}
where $\epsilon_{ij}=(\vert x_i\vert +1)(\vert x_1\vert+\ldots+
\vert x_{i-1}\vert+i-1)+(\vert x_j\vert +1)(\vert x_1\vert
+\ldots+\vert x_{j-1}\vert+j-1)-(\vert x_i\vert +1)(\vert x_j\vert +1)$.
\end{definition}

Now suppose $(\mathrm C_\bullet, b, B)$ is a mixed complex.
Denote by $\mathrm{CC}^-_{\bullet}(\mathrm C_\bullet)$
the negative cyclic complex of $\mathrm C_\bullet$.
Then we have a short exact sequence
$$0\longrightarrow
u\cdot \mathrm{CC}_{\bullet+2}^{-} (\mathrm C_\bullet)
\stackrel{\iota}\longrightarrow
\mathrm{CC}_\bullet^{-} (\mathrm C_\bullet)
\stackrel{\pi}\longrightarrow
\mathrm{C}_\bullet \longrightarrow 0,$$
where
$\iota: u\cdot \mathrm{CC}_{\bullet+2}^{-} (\mathrm C_\bullet)
\to
\mathrm{CC}_\bullet^{-} (\mathrm C_\bullet)
$
is the embedding
and
$$
\pi:\mathrm{CC}_\bullet^{-}(\mathrm{C}_{\bullet}) \to
\mathrm{C}_\bullet ,\quad
\displaystyle\sum_i x_i\cdot u^i\mapsto x_0
$$
is the projection.
It induces functorially a long exact sequence
$$
\cdots\longrightarrow
\mathrm{HC}_{\bullet+2}^{-} (\mathrm C_\bullet)
\longrightarrow
\mathrm{HC}_\bullet^{-} (\mathrm C_\bullet)
\stackrel{\pi_*}\longrightarrow
\mathrm{HH}_\bullet (\mathrm C_\bullet)
\stackrel{\beta}\longrightarrow
\mathrm{HC}_{\bullet+1}^{-} (\mathrm C_\bullet)
\longrightarrow\cdots,
$$
where $\mathrm{HH}_\bullet(\mathrm C_\bullet)$
and $\mathrm{HC}^-_\bullet(\mathrm C_\bullet)$
are the Hochschild and negative cyclic homology of $\mathrm C_\bullet$ respectively.
(Recall that $\mathrm{HH}_\bullet(\mathrm C_\bullet)$ is just the $b$-homology of
$\mathrm C_\bullet$.)
The main result obtained in \cite{CEL} is the following.

\begin{lemma}[\cite{CEL} Theorem 1.1]\label{lemma:gravity}
Suppose $(\mathrm C_\bullet, b, B)$ is a mixed complex.
If $\mathrm{HH}_\bullet(\mathrm C_\bullet)$
has a Batalin-Vilkovisky algebra structure
such that $B$ is the generator of the Gerstenhaber bracket, then
the following sequence of maps
$$\begin{array}{cccl}
\{-,\cdots,-\}:&(\mathrm{HC}^{-}_\bullet(\mathrm C_\bullet))^{\otimes n}
&\longrightarrow& \mathrm{HC}^{-}_\bullet(\mathrm C_\bullet)\\
&(x_1,\cdots, x_n)&\longmapsto&(-1)^{\varepsilon_n}\beta\big(\pi_*(x_1)
\bullet \pi_*(x_2)\bullet\cdots\bullet
\pi_*(x_n)\big),\quad n=2, 3,\cdots
\end{array}
$$
where $\varepsilon_n=(n-1)|x_1|+(n-2)|x_2|+\cdots+|x_{n-1}|$
and $\bullet$ is the product on the Hochschild homology (coming from
the Batalin-Vilkovisky algebra structure),
gives on $\mathrm{HC}^{-}_\bullet(\mathrm C_\bullet)$ a gravity algebra structure.
\end{lemma}

In what follows, we shall also study the cyclic cohomology of an associative
and Poisson algebra. Suppose $(\mathrm C^\bullet, \delta, B^*)$ is a mixed complex
with degrees of $\delta$ and $B^*$ being $1$ and $-1$ respectively;
in order to distinguish,
we would call this type of mixed complex in what follows {\it mixed cochain complex},
and call the usual mixed complex, like $(\mathrm C_\bullet, b, B)$ above, {\it
mixed chain complex}.
By our convention, the cyclic cohomology of a mixed cochain complex $(\mathrm C^\bullet, \delta, B^*)$,
denoted by
$\mathrm{HC}^\bullet(\mathrm C^\bullet)$,
is the {\it cohomology} of the negative cyclic complex
of the mixed chain complex obtained from $(\mathrm C^\bullet, \delta, B^*)$
by negating the gradings. Thus the cyclic cohomology
is essentially the same as the negative cyclic homology.

Now suppose $A$ is a Poisson algebra and respectively $A^!$
is a Frobenius Poisson
algebra, both with diagonalizable modular vector fields. In the previous section
we have shown that $(\mathrm{CP}_\bullet^0(A, A_\nu), \partial_\nu, d)$
is a mixed chain complex
and $(\mathrm{CP}^\bullet_0(A^!, A_{\nu^!}^{\ac}), \partial_{\nu^!}, d^*)$
is a mixed cochain complex.

\begin{definition}\label{def:negativecyclicPoisson}
Suppose $A$ is a Poisson algebra
with diagonalizable modular vector field $\nu$.
The negative cyclic homology
of the mixed complex
$$
(\mathrm{CP}_\bullet^0(A, A_\nu), \partial_\nu, d)
$$
is called the {\it negative cyclic Poisson homology}
of $A$, and is denoted by $\mathrm{PC}^{-}_{\bullet}(A)$.
Similarly,
suppose $A^!$
is a Frobenius Poisson
algebra with
diagonalizable modular vector field $\nu^!$.
The cyclic cohomology of
$(\mathrm{CP}^\bullet_0(A^!, A_{\nu^!}^{\ac}),
\partial_{\nu^!}, d^*)$
is called the {\it cyclic Poisson cohomology} of
$A^!$, and is denoted by $\mathrm{PC}^\bullet(A^!)$.
\end{definition}

\begin{theorem}\label{thm:gravityiso}
Suppose $A=\mathbb R[x_1,\cdots, x_n]$.
For a quadratic Poisson structure on $A[\![\hbar]\!]$ with diagonalizable modular vector field,
the following diagram
\begin{equation}\label{diag:commgravity}
\begin{split}
\xymatrixcolsep{3pc}
\xymatrix{
\mathrm{PC}^{-}_\bullet(A[\![\hbar]\!])\ar[r]^-{\cong}\ar[d]^-{\cong}&\mathrm{PC}^\bullet(A^![\![\hbar]\!])\ar[d]^-{\cong}\\
\mathrm{HC}^{-}_\bullet(A_\hbar)\ar[r]^-{\cong}&\mathrm{HC}^\bullet(A^!_\hbar)
}
\end{split}
\end{equation}
commutes, where $\mathrm{HC}_\bullet^{-}(A_\hbar)$
and $\mathrm{HC}^\bullet(A_\hbar^!)$ are the usual negative cyclic homology
and the cyclic cohomology of $A_\hbar$ and $A_\hbar^!$ resepctively.
Moreover, after the degrees shifted down by $n$, the above is a commutative diagram of isomorphisms of gravity algebras.
\end{theorem}

\begin{proof}By transporting the Batalin-Vilkovisky algebra
structures in Theorem \ref{thm:BViso} to the right hand side
of the diagram \eqref{maindiagram},
the theorem then follows from Lemma \ref{lemma:gravity}.
\end{proof}


\begin{thebibliography}{100}

\bibitem{AS} M. Artin and W.F. Schelter, {\it Graded algebras of global dimension 3.}
Adv. Math. 66 (1987), 171--216.

\bibitem{BZ}
K. Brown and J.J. Zhang,
{\it Dualising complexes and twisted Hochschild (co)homology for noetherian Hopf algebras.}
J. Algebra 320 (2008) 1814--1850.

\bibitem{Brylinski}J.L. Brylinski,
{\it A differential complex for Poisson manifolds.} J. Differential Geom. 28 (1988) 93--114.

\bibitem{BW}H. Bursztyn and S. Waldmann, {\it Bimodule deformations,
Picard groups and contravariant connections. }
K-Theory 31 (2004), no. 1, 1--37.


\bibitem{Calaque+}D. Calaque, G. Felder, A. Ferrario and C. A. Rossi,
{\it Bimodules and branes in deformation quantization.}
Compositio Math. 147 (2011) 105--160.

\bibitem{CF}
A.S. Cattaneo and G. Felder,
{\it Relative formality theorem and quantisation of coisotropic submanifolds.}
Adv. Math.  208 (2007) 521--548.

\bibitem{Chemla}S. Chemla,
{\it Formality theorem with coefficients in a module.}
Transformation Groups, Vol. 13, No. 1, 2008, pp. 91--123.



\bibitem{CCEY}X. Chen, Y. Chen, F. Eshmatov and S. Yang,
{\it Poisson cohomology, Koszul duality, and Batalin-Vilkovisky algebras.} J. Noncomm. Geom.,
15 (2021), 889--918.

\bibitem{CEL}
X. Chen, F. Eshmatov and L. Liu,
{\it Gravity algebra structure on the negative cyclic homology
of Calabi-Yau algebras.} J. Geom. Phys. 147 (2020), 103522, 19 pp.



\bibitem{Dolgushev}
V.A. Dolgushev,
{\it The Van den Bergh duality and the modular symmetry of a Poisson variety.}
Selecta Math. (N.S.) 14 (2009) 199--228.



\bibitem{Getzler1}E. Getzler,
{\it Batalin-Vilkovisky algebras and two-dimensional topological field theories.}
Comm. Math. Phys. 159 (1994) 265--285.

\bibitem{Getzler2}E. Getzler,
{\it Two-dimensional topological gravity and
equivariant cohomology.} Comm. Math. Phys. 163 (1994), no. 3, 473--489.

\bibitem{Ginzburg}V. Ginzburg, {\it Calabi-Yau algebras}, arXiv: 0612139v3.




\bibitem{Kontsevich}M. Kontsevich,
{\it Deformation quantization of Poisson manifolds.}
Lett. Math. Phys. 66 (2003) 157--216.

\bibitem{Koszul}
J.-L. Koszul,
{\it Crochet de Schouten-Nijenhuis et cohomologie.}
Ast\'{e}risque, num\'{e}ro hors s\'{e}rie, 257--271 (1985).


\bibitem{KK}
N. Kowalzig and U. Kr\"ahmer,
{\it Batalin-Vilkovisky structures on Ext and Tor.}
J. Reine Angew. Math. 697 (2014), 159--219.



\bibitem{Lambre}
T. Lambre,
{\it Dualit\'e de Van den Bergh et Structure de Batalin-Vilkovisky sur les alg\`ebres
de Calabi-Yau.}
J. Noncom. Geom. 3 (2010) 441--457.

\bibitem{LZZ} T. Lambre, G.D. Zhou and A. Zimmermann,
{\it The Hochschild cohomology ring
of a Frobenius algebra with semisimple Nakayama automorphism is
a Batalin-Vilkovisky algebra.} J. Algebra 446(2016), 103--131.

\bibitem{LR}
S. Launois and L. Richard,
{\it Twisted Poincar\'{e} duality for some quadratic Poisson algebras.}
Lett. Math. Phys. 79 (2007) 161--174.

\bibitem{LGPV}
C. Laurent-Gengoux, A. Pichereau and P. Vanhaecke,
{\it Poisson structures.}
Grundl. Math. Wiss. 347. Springer, Heidelberg, 2013.

\bibitem{Lichnerowicz}
A. Lichnerowicz,
{\it Les vari\'et\'es de Poisson et leurs alg\`ebres de Lie associ\'ees.}
J. Differential Geom.  12 (1977) 253--300.


\bibitem{Liu}
L. Liu,
{\it Koszul duality and the Hochschild cohomology of
Artin-Schelter regular algebras.} Homology,
Homotopy and Applications, vol. 22(2), 2020, pp.181--202.

\bibitem{LWZ}J. L\"u, X. Wang and G. Zhuang,
{\it Homological unimodularity and Calabi-Yau condition for Poisson algebras.}
Lett Math Phys (2017) 107:1715--1740.

\bibitem{LWW}
J. Luo, S.-Q. Wang and Q.-S. Wu,
{\it Twisted Poincar\'{e} duality between Poisson homology and Poisson cohomology.}
J. Algebra 442 (2015) 484--505.


\bibitem{LWW2}J. Luo, S.-Q. Wang and Q.-S. Wu,
{\it Frobenius Poisson algebras.}
Front. Math. China 2019, 14 (2): 395--420.

\bibitem{MT}D. Manchon and T. Torossian,
{\it Cohomologie tangente et cup-produit pour la quantification de Kontsevich.}
Ann. Math. Blaise Pascal 10 (1), 75--106 (2003).


\bibitem{RRZ}
M. Reyes, D. Rogalski and J.J. Zhang,
{\it Skew Calabi-Yau algebras
and Homological identities.} Adv. Math.
volume 264, (2014), 308--354.


\bibitem{Shoikhet}
B. Shoikhet,
{\it A proof of the Tsygan formality conjecture for chains.} Adv. Math.
179 (1), 7--37 (2003).

\bibitem{Shoikhet1}
B. Shoikhet,
{\it Koszul duality in deformation quantization and Tamarkin's approach to Kontsevich formality.}
Adv. Math.  224 (2010) 731--771.

\bibitem{Smith}
S.P. Smith, {\it Some finite dimensional algebras related to elliptic curves}, in
Representation theory of algebras and related topics (Mexico City, 1994), 315--348,
CMS Conf. Proc. 19, Amer. Math. Soc., Providence, RI, 1996.


\bibitem{TT}
D. Tamarkin and B. Tsygan,
{\it The ring of differential forms in noncommutative calculus.}
Proc. Sympos. Pure Math., 73, Amer. Math. Soc., Providence, RI, 2005.


\bibitem{VdB96}
M. Van den Bergh, Existence theorems for dualizing complexes over
non-commutative graded and filtered rings, J. Algebra 195 (1997), no. 2, 662--679.

\bibitem{VdB97}
M. Van den Bergh,
{\it A relation between Hochschild homology and cohomology for Georenstein rings.}
Proc. Amer. Math. Soc. 126 (1998) 1345--1348. Erratum 130 (2002) 2809--2810.

\bibitem{WWZZ}S.-Q. Wang, Q.-S. Wu, G. Zhou
and C. Zhu, {\it The modular derivation and Batalin-Vilkovisky
structure of a Frobenius Poisson algebra}, unpublished manuscript.

\bibitem{Weinstein}A. Weinstein,
{\it The modular automorphism group of a Poisson manifold.}
J. Geom. Phys. 23 (1997) 379--394.

\bibitem{Xu}P. Xu,
{\it Gerstenhaber algebras and BV-algebras in Poisson geometry.}
Comm. Math. Phys. 200 (1999) 545--560.

\bibitem{ZVOZ}
C. Zhu, F. Van Oystaeyen and Y. Zhang,
{\it On (co)homology of Frobenius Poisson algebras.}
J. K-Theory  14 (2014) 371--386.


\bibitem{Zhu}C. Zhu,
{\it Twisted Poincar\'{e} duality for Poisson homology
and cohomology of affine Poisson algebras.}
Proc. Amer. Math. Soc.  143 (2015) 1957--1967.



\end{thebibliography}
\end{document}